\documentclass[12pt,a4paper,final,DIV=13]{scrartcl}
\usepackage[utf8]{inputenc}
\usepackage[T1]{fontenc}
\usepackage{lmodern}

\usepackage{matchsizex}

\usepackage{doi}
\usepackage[numbers]{natbib}
\usepackage{subcaption}

\makeatletter
\let\@nohyperfootnotemark\@footnotemark
\let\@nohyperfootnotetext\@footnotetext
\def\nohyperfootnote{\@ifnextchar[\@xfootnote{\stepcounter\@mpfn
     \protected@xdef\@thefnmark{\thempfn}%
     \@nohyperfootnotemark\@nohyperfootnotetext}}
\makeatother

\newcommand\blfootnote[1]{%
  \begingroup
  \renewcommand\thefootnote{}\nohyperfootnote{#1}%
  \addtocounter{footnote}{-1}%
  \endgroup
}

\usepackage[defthms,thmswithin=section,bib=false,%
  thmbm=false,sfthm=true,screen=true]{arw}

\theoremstyle{definition}

\makeatletter
\newcommand{\fakephantomsection}{%
	\Hy@MakeCurrentHref{\@currenvir.\the\Hy@linkcounter}
	\Hy@raisedlink{\hyper@anchorstart{\@currentHref}\hyper@anchorend}%
}
\makeatother

\newcommand\Pp{\mathcal{P}}

\newcommand\pp{\mathbf{p}}

\newcommand\Zz{\mathcal{Z}}
\newcommand\Zb{\mathbf{Z}}
\newcommand\Yy{\mathcal{Y}}

\newcommand\Mmp{\mathcal{M}_p}
\newcommand\WW{\mathcal{W}}

\newcommand\mg[2]{W(#1,#2)}

\newcommand\lP{\mathbb{P}} 
\newcommand\lPn{\mathbb{P}^{(b_n)}}
\newcommand\lQ{\mathbb{P}_\omega}
\newcommand\lQn{\mathbb{P}_{\omega}^{(b_n)}}
\newcommand\lQt{\mathbb{P}_\omega}
\newcommand\lQtn{\mathbb{P}_{\omega}^{(b_n)}}
\newcommand\lE{\mathbb{E}}
\newcommand\lEn{\mathbb{E}^{(b_n)}}
\newcommand\lEQ{\mathbb{E}_\omega}
\newcommand\lEQn{\mathbb{E}_{\omega}^{(b_n)}}
\newcommand\lEQt{\mathbb{E}_{\omega}}
\newcommand\lEQtn{\mathbb{E}_{\omega}^{(b_n)}}
\newcommand\sP{\mathbf{P}}
\newcommand\sE{\mathbf{E}}

\newcommand\lQD{\mathbb{Q}_{\omega}}

\newcommand\lQc{\mathbb{P}_{\bar\omega}}
\newcommand\lEQc{\mathbb{E}_{\bar \omega}}
\newcommand\lQDc{\mathbb{Q}_{\bar{\omega}}}

\newcommand\Indic[1]{\Ind_{\{#1\}}}
\newcommand\FFs{\hat\FF}

\newcommand\card{\#}

\newcommand\tree{\mathcal{U}}
\newcommand\Ess{\mathcal{E}}
\newcommand\MA{W} 
\newcommand\MD{\partial W} 

\newcommand\ws{\bar \omega}

\DeclareMathOperator{\Anc}{Anc}
\DeclareMathOperator{\Desc}{Desc}

\DeclareMathOperator{\dom}{dom}

\DeclarePairedDelimiterX\ip[2]{\langle}{\rangle}{#1,#2}

\DeclareNormlike\oqv[] 
\DeclareNormlike\pqv\langle\rangle 

\addto\extrasbritish{%

}

\theoremstyle{definition}
\newsstheorem{conj}{Conjecture}

\title{Probability tilting of compensated fragmentations}
\author{%
  Quan Shi\footnote{University of Mannheim, Mannheim, 68131, Germany; \email{quanshi.math@gmail.com}}%
  \and
  Alexander R. Watson\footnote{University of Manchester, Manchester, M13 9PL, UK; \email{alex.watson@manchester.ac.uk}}}

\begin{document}

  \maketitle
  
  \begin{abstract}
    Fragmentation processes are part of a broad class of models
    describing the evolution of a system of particles which
    split apart at random.
    These models are widely used in biology, materials science and nuclear
    physics, and their asymptotic behaviour at large times is interesting
    both mathematically and practically.
    The spine decomposition is a key tool in its study.
    In this work, we consider the class of compensated
    fragmentations, or homogeneous growth-fragmentations, recently defined by Bertoin.
    We give a complete spine decomposition of these processes in terms of 
    a L\'evy process with immigration, and
    apply our result to study the asymptotic properties of the
    derivative martingale.%
    \blfootnote{%
    }%
    \blfootnote{%
    }%
  \end{abstract}

  \section{Introduction}
  \label{s:intro}
    
  Fragmentation processes offer a random model for particles which break apart
  as time passes. Informally, we imagine a single particle, characterised by its mass,
  which after some random time splits into two or more daughter particles, distributing its
  mass between them according to some law. The new particles act independently
  of one another and evolve in the same way.
  Variants of such processes have been studied over many years, with applications
  across the natural sciences \cite{BCP-frag,
  BA-cell-growth,Cx-size-dist}.
  One large class of fragmentation models,
  encompassing the so-called homogeneous fragmentation processes,
  has been particularly successful, and a comprehensive discussion can be found
  in the book of \citet{Ber-fc}. 
  
  Compensated fragmentation processes were defined by \citet{Ber-cfrag}
  as a generalisation of
  homogeneous fragmentations,
  and permit high-intensity fragmentation and Gaussian fluctuations of the sizes of fragments.
  The processes arise as the limits of homogeneous fragmentations
  under dilation \cite[Theorem 2]{Ber-cfrag},
  and may also be thought of as being related to
  a type of branching Lévy process, for which the branching occurs at the jump times of the process.
  From this viewpoint, they may be regarded as the simplest example in the class of 
  so-called Markovian growth-fragmentation processes \cite{BeMGF},
  and for this reason they are sometimes called \emph{homogeneous} growth-fragmentation processes.
  Other examples in the 
  class of Markovian growth-fragmentations
  can be obtained by slicing planar random maps with boundary,
  as discovered by \citet{BBCK-maps}, or by considering
  the destruction of an infinite recursive tree, as in \citet{BB-ou}.
  
  \skippar
  The main purpose of this work is to give a complete spine decomposition
  for compensated fragmentation processes.  
  This is motivated by the many applications that such decompositions have
  found in proving powerful results
  across the spectrum of branching process models.
  Since the foundational work of \citet{LPP-LlogL}
  on `conceptual' proofs of the $L \log L$ criterion for Galton--Watson processes,
  a large literature has emerged, of which we offer here only a selection,
  focusing on the applications we have in mind.  
  
  In the context of branching random walks, the spine decomposition has
  been used to prove martingale convergence theorems and to study
  the asymptotics, fluctuations and genealogy of the largest particle;
  see \cite{Shi-BRW} for a detailed monograph with historical references.
  For branching Brownian motion, spine techniques were
  used by \citet{CR-kpp} to describe asymptotic presence
  probabilities, and by
  \citet{Kyp-FKPP} and \citet{RY-dm}
  to study solutions of reaction-diffusion equations of
  Fisher--Kolmogorov--Petrowski--Piscounov (FKPP) type.
  In the context of superprocesses, we mention the study of strong laws of large numbers
  by \citet{EKW-slln}, which also contains a thorough review of the literature.
  
  Spine techniques have lent themselves well to the study of
  homogeneous (pure) fragmentation processes.
  Convergence theorems were proved by \citet{BR-disc}, and
  the decomposition was used by \citet{Haa-loss} to study the fragmentation equation,
  \citet{HKK-lln} for the proof of strong laws of large numbers,
  and \citet{BHK-FKPP} to look at solutions of FKPP equations.
  Returning to the topic of growth-fragmentation processes,
  \citet[\S 4]{BBCK-maps} established a spine decomposition and used it in order to study
  certain random planar maps, and the results presented in this
  article overlap with theirs under certain parameter choices (see \autoref{r:mto}\ref{i:mto}.)
  \citet[\S 3.2]{BerSte} gave an explicit decomposition
  for compensated fragmentation processes in
  the case of finite fragmentation rate and applied it to
  the phenomenon of local explosion, and
  \citet[\S 6]{BW-gfe} made implicit use of a
  spine decomposition in studying the growth-fragmentation equation.
  Since the first appearance of this work, \citet[Lemma~2.3]{BM-mcbLp} 
	gave a version of the spine decomposition in the more general setting of branching L\'evy processes (see \autoref{r:BM}\ref{i:BM}.)

  \medskip\noindent  
  Our object of study is the compensated fragmentation process
  $\mathbf{Z} = (\mathbf{Z}(t), t\ge 0)$,
  where $\mathbf{Z}(t) = (Z_1(t),Z_2(t),\dotsc)$ is an element of
  $\ell^{2\downarrow} = \{ \mathbf{z} = (z_1,z_2,\dotsc) : z_1 \ge z_2 \ge \dotsb \ge 0, \, \sum_{i=1}^\infty z_i^2<\infty \}$.
  The values $Z_1(t),Z_2(t),\dotsc$
  are regarded as the ranked sizes of \emph{fragments}
  as seen at time $t$.
  Unless otherwise specified, we will assume that $\mathbf{Z}(0) = (1,0,\dotsc)$.
  
  The law of $\mathbf{Z}$ is characterised by a triple $(a,\sigma,\nu)$
  of \emph{characteristics},
  where $a \in \RR$, $\sigma \ge 0$ and $\nu$ is a non-trivial measure on the space
  \[
  \Pp =\mbigl\{ \mathbf{p} = (p_1,p_2,\dotsc) : p_1\ge p_2 \ge \dotsb\ge 0, \sum_{i=1}^\infty p_i \le 1 \mr\},
  \]
  satisfying the moment condition 
  \begin{equation}\label{e:nu}
  	\int_{\Pp} (1-p_1)^2 \,\nu(\dd\pp) < \infty .
  \end{equation}
  Loosely speaking, $a$ describes deterministic growth or decay of the fragments and
  $\sigma$ describes the magnitude of Gaussian fluctuations
  in their sizes. The measure $\nu$ is called the
  \emph{dislocation measure}, and $\nu(\dd\pp)$ represents the rate at which
  a fragment of size $x$ splits into a cloud of particles of sizes $xp_1,xp_2,\dotsc$.
  
  The connection between $\mathbf{Z}$ and the triple is given by
  the \emph{cumulant} $\kappa$, 
  which is defined by the equation
  $e^{t \kappa(q)} = \lE\mbigl[\sum_{i\ge 1} Z_i(t)^q\mr]$.
  It is given by the following expression, akin to the L\'evy--Khintchine formula for
  L\'evy processes:
  \begin{equation}\label{e:kappa}
  	\kappa(q) 
  	= 
  	\frac{1}{2}\sigma^2q^2 + aq 
  	+ \int_{\Pp} \mBigl[ \sum_{i\ge 1} p_i^q - 1 + (1-p_1)q \mr] \, \nu(\dd \pp), \qquad q\in \RR.  \end{equation}
  The function $\kappa$ takes values in $\RR\cup\{\infty\}$.
  We regard
  $\dom \kappa := \{ q\in \RR : \kappa(q) < \infty \}$ as the function's domain. 
  Condition \eqref{e:nu} entails that 
  \begin{equation}\label{eq:dom}
  q \in \dom \kappa \quad \text{if and only if}\quad  \int_{\Pp}  \sum_{i\ge 2} p_i^q  \, \nu(\dd \pp) <\infty ,
  \end{equation}
  and that $[2,\infty) \subset \dom \kappa$. 
  One notable property of $\kappa$ is that it is strictly convex and smooth on the interior
  of $\dom\kappa$.
  
  If the measure $\nu$ satisfies the stronger moment condition
  $\int_{\Pp} (1-p_1) \, \nu(\dd\pp)<\infty$,
  and $\sigma = 0$, 
  then $\kappa$ is the cumulant of a homogeneous fragmentation process
  $\mathbf{Z}$ in the sense of \cite{Ber-fc}, with additional deterministic exponential
  growth or decay.
  
  \skippar
  We shall prove a spine decomposition for $\mathbf{Z}$
  under a change of measure.
  In particular, 
  for $\omega \in \dom \kappa$, we define the \emph{(exponential) additive martingale} $\mg{\omega}{\cdot}$
  as follows:
  \begin{equation*}
  	\mg{\omega}{t} = e^{-t\kappa(\omega)} \sum_{i\ge 1} Z_i(t)^\omega, \wh t \ge 0.
  \end{equation*}
  Since this is a unit-mean martingale
  (see the forthcoming \autoref{l:Zbar-bp}),
  we may define a new, `tilted' probability measure
  $\lQ$, as follows. Fix $t \ge 0$, and let $A$ be an event depending only on the
  path of $\mathbf{Z}$ up to time $t$.
  Then, define
  \begin{equation*} \lQ(A) = \lE[ \Ind_A \mg{\omega}{t} ] . \end{equation*}
  Our first main result is \autoref{t:mto}, in which we show
  that under $\lQ$, the process $\mathbf{Z}$
  may be regarded as
  the exponential of a single spectrally negative Lévy process (the \emph{spine})
  with Laplace exponent
  $\kappa(\cdot + \omega) - \kappa(\omega)$, onto whose jumps are grafted independent
  copies of $\mathbf{Z}$ (under the original measure $\lP$).
  This is the \emph{spine decomposition}, also known as a
  \emph{full many-to-one theorem}.
  
  In order to illustrate the power of this spine decomposition,
  we study the \emph{derivative martingale} associated with $\mathbf{Z}$. For $\omega$ in the interior of $\dom \kappa$, this is defined by
  \begin{equation}\label{}
  	\MD(\omega, t)
  	= \frac{\partial }{\partial \omega}\mg{\omega}{t}
  	= e^{-t \kappa(\omega)}
  	\sum_{i\ge 1} \mbigl(  -t \kappa'(\omega) + \log Z_i(t) \mr)
  	Z_i(t)^\omega, \qquad t\geq 0.
  \end{equation}
  Since this martingale can take both positive and negative values,
  it is not immediately obvious whether its limit as $t\to\infty$
  exists.
  
  Using our decomposition, we prove our second main result,
  \autoref{t:DerMart}, which states that the derivative martingale converges
  to a strictly negative limit under certain conditions.
  This limit is closely related to the process representing the largest
  fragment of the compensated fragmentation.
  Our theorem is the counterpart of results on the asymptotics of the derivative martingale
  which have been found in the context
  of homogeneous (pure) fragmentation processes \cite{BR-disc},
  branching random walks \cite{BK-mc,Shi-BRW}
  and branching Brownian motion \cite{Kyp-FKPP}.
  In the case of compensated fragmentation processes,
  \citet{Dadoun:agf} studied the discrete-time skeletons of the
  derivative martingale via a branching random walk, and used their convergence to obtain
  asymptotics for the largest fragment.
  Our work complements and extends this by showing the almost sure
  convergence of the martingale in continuous time and
  showing that the expectation of the terminal value is infinite;
  we also obtain somewhat weaker conditions.

  
  \skippar
  This work lays the foundations for future research in two principal directions.
  The first concerns more general Markovian growth-fragmentations, and in particular
  we anticipate that it should be possible to extend
  the spine decomposition to
  growth-fragmentations based on generalised Ornstein--Uhlenbeck processes,
  as studied in \cite{Shi-ougf,BB-ou}.
  The second concerns applications for the homogeneous processes studied here.
  Our asymptotic
  results for the derivative martingale
  may be used to study the size of the largest fragment and
  the existence and uniqueness of travelling wave solutions to
  FKPP equations, much as in \cite{BHK-FKPP}.
  
  \paragraph{Organisation of this work} In \autoref{s:cf}, we give a rigorous
  definition of the branching L\'evy process, outlining the truncation
  argument of \cite{Ber-cfrag} and simultaneously define a new labelling
  scheme for its particles.
  In \autoref{s:backward}, we consider the measure $\lQ$
  just presented, additionally distinguishing a single particle
  by picking from those particles alive at time $t$ in a size-biased way.
  In \autoref{s:forward}, we present a complete construction of a Markov process
  with a single distinguished particle, which we claim gives the law of the
  process $\mathbf{Z}$ with distinguished particle under $\lQ$; this claim is then
  proven in \autoref{s:fb}.
  Finally, we discuss the asymptotic properties of the derivative martingale in
  \autoref{s:dm}.

  \section{The branching L\'evy process}
  \label{s:cf}
 
  Our goal in this section is to establish a genealogical structure for the
  compensated fragmentation
  process $\mathbf{Z}$, that is to represent it as a random infinite marked tree.
  This is what allows us to study the spine decomposition. 
  To be specific, we will define a family of L\'evy processes, $(\Zz_u, u\in \tree)$,
  labelled by the nodes
  of a tree $\tree$. For $t\ge 0$, let $\tree_t$ be the set of
  individuals present at time $t$. We will be able to list the elements of
  $\tree_t$ by $u_1,u_2,\dotsc$ such that $\Zz_{u_1}(t) \ge \Zz_{u_2}(t) \ge \dotsb$.
  The compensated fragmentation process at time $t$ is then given by
  \[ \mathbf{Z}(t) = (\exp(\Zz_{u_1}(t)),\exp(\Zz_{u_2}(t)), \dotsc). \]
  We also define a related point measure-valued process, called the \emph{branching L\'evy process}:
  \[\Zz(t) = \sum_{u\in \tree_t} \delta_{\Zz_u(t)}. \]
 One can easily recover the compensated fragmentation process $\mathbf{Z}$ from $\Zz$.
  Therefore, for convenience, we shall always work with $\Zz$ from now on and state all our results in terms of $\Zz$.

  \subsection{L\'evy processes}
  \label{s:LP}

  Since our main object of study is a branching L\'evy process, it is unsurprising that
  L\'evy processes play a key role. We give a short summary of the relevant definitions
  and properties.

  A stochastic process $\xi = (\xi(t), t\ge 0)$
  under a probability measure $\sP$ is called a \emph{L\'evy process} if it has
  stationary, independent increments and c\`adl\`ag paths,
  and satisfies $\xi(0) = 0$ almost surely. The process
  $\xi$ is said to be \emph{spectrally
  negative} if the only points of discontinuity of its paths are negative jumps.
  The usual way to characterise the law of such a process is through its Laplace
  exponent; this is a function $\Psi \from \RR \to \RR\cup\{\infty\}$, such that
  for every $t\ge 0$, $\sE[ e^{q\xi(t)} ] = e^{t\Psi(q)}$. It is well-known that
  $\Psi$ satisfies the so-called \emph{L\'evy--Khintchine formula}, as follows:
  \begin{equation}
    \label{e:lk}
    \Psi(q) = \frac{1}{2}\gamma^2 q^2 + \mathtt{a}q + \int_{(-\infty,0)} \mbigl[ e^{qx} - 1 - qx\Indic{x>-1} \mr] \, \Pi(\dd x), \qquad q\in\RR,
  \end{equation}
  and $\Psi(q) < \infty$ if $q\ge 0$.
  Here, $\mathtt{a} \in \RR$ is called the \emph{centre} of $\xi$, $\gamma \ge 0$ is the
  \emph{Gaussian coefficient}, and $\Pi$ is a measure, called the \emph{L\'evy measure},
  on $(-\infty,0)$, which satisfies the moment condition
  $\int_{(-\infty,0)} \min \{ 1, x^2\} \, \Pi(\dd x) <\infty$.

  The classification of L\'evy processes is made
  more precise by the \emph{L\'evy--It\^o decomposition}, which we now describe.
  Let $\mathtt{M}$ be a Poisson random measure on $[0,\infty)\times (-\infty,0)$
  with intensity measure $\text{Leb}\times\Pi$.
  Let $B = (B(t) , t \ge 0)$ be a standard Brownian motion
  independent of $\mathtt{M}$.
  Then, a L\'evy process $\xi$ with Laplace exponent $\Psi$ can
  be constructed as:
  \[
    \xi(t) = \gamma B(t) + \mathtt{a}t
    + \int_{[0,t]\times (-\infty,-1]} x \, \mathtt{M}(\dd s, \dd x)
    + \lim_{\epsilon\downto 0} \int_{[0,t]\times (-1,-\epsilon)} x\, \bigl[ \mathtt{M}(\dd s, \dd x) - \dd s \Pi(\dd x)\bigr]
    ,
  \]
  and the limit of compensated small jumps which appears as the last term
  is guaranteed to exist in the $L^2$ sense.
  We refer to the measure $\mathtt{M}$ as the \emph{jump measure} of $\xi$.

  Standard works on this class of processes are the books \cite{Ber-Levy,Kyp2,Sato}.
  We mention here only one additional feature which will be useful in our
  study of the spine decomposition.
  If $\Psi(\omega)<\infty$, then the process
  $M_\omega(t) = e^{\omega \xi(t) - t\Psi(\omega)}$ is a unit-mean martingale
  under $\sP$ for the natural filtration of $\xi$,
  and if we define a new measure $\sP_{\omega}$ via
  \[
    \sP_{\omega}(A) \coloneqq \sE[ \Ind_A M_\omega(t) ] ,
    \qquad A \in \sigma(\xi(s),\, s\le t), \; t \ge 0,
  \]
  then $\xi$ under $\sP_{\omega}$ is a spectrally negative 
  L\'evy process \cite[\S 8.1]{Kyp2}. Its Laplace exponent is the function $\Ess_\omega \Psi$
  defined by
  \[ \Ess_{\omega}\Psi(q) \coloneqq \Psi(q+\omega) - \Psi(\omega), \qquad q\ge 0. \]
  This new process has
  centre
  $\mathtt{a}_\omega \coloneqq \mathtt{a} + \gamma^2 \omega + \int_{(-1,0)} x(e^{\omega x}-1)\, \Pi(\dd x)$,
  Gaussian coefficient $\gamma$,
  and L\'evy measure $\Pi_\omega(\dd x) \coloneqq e^{\omega x} \Pi(\dd x)$.
  The function $\Ess_\omega \Psi$ is referred to as the
  \emph{Esscher transform} of $\Psi$.

  \subsection{Construction and truncation of the branching L\'evy process}  
  \label{s:construct-truncate}\label{s:labels}

  In this section, we give a rigorous definition of the branching Lévy process.
  Our presentation is inspired by \citet{Ber-cfrag}, and the main idea is first to define,
  given a sequence of numbers $b_n \ge 0$,
  a collection of \emph{truncated} processes $\bar{\Zz}^{(b_n)}$
  representing the positions, and attached labels, of particles which do not
  land `too far' (i.e., at a distance greater than $b_n$)
  from their parent.
  This is necessary since the rate of fragmentation is, in general, infinite.
  These processes will be constructed such that they are consistent with one another,
  in a sense
  which will shortly be made precise, and such that taking $n\to\infty$ reveals all of the
  particles. The main innovation compared to \cite{Ber-cfrag} is the
  inclusion of labels for the particles,
  and this is what allows us to study the spine decomposition.
  
 Roughly speaking, we will use a Crump-Mode-Jagers type labelling scheme, in which
 the closest of the `offspring' of a particle at each branching event retains the parent's identity;
 see  \cite{Jag-gbp} for a discussion of this so-called `general branching process' framework.
 Our system is reminiscent of the
 one adopted in \cite{BeMGF}, which also uses immortal particles
 with labels based on the size of the jumps, but for which the labels are purely
 generational. 
 With this construction, one will also be able to define stopping lines; a particular case (the stopped martingale in \autoref{s:sm}) will play an essential role in the proof of \autoref{t:DerMart}.   
 We mention here also an alternative approach to the genealogy
 by \citet{BM-bLp}, based upon a restriction to dyadic rational times,
 which is of quite a different style.

  Readers who are already familiar with the construction
  of \cite{Ber-cfrag} may wish to skip this
  section on first reading, and simply assume the existence of
  a set of particle labels which is consistent under truncation.
  
  \skippar
  Let us introduce some notation.
  The set of labels will be given by
  $\tree = \cup_{j\ge 0} (\NN^3)^j$, where we use the convention
  $(\NN^3)^0 = \{\varnothing\}$,
  and we will denote elements of this set in the following way: if $u_i \in \NN^{3}$
  for $i=1,\dotsc,I$, then we will write $(u_1,u_2,\dotsc,u_I)$ as $u_1 u_2 \dotsb u_I$.  
  The label $\varnothing$ represents the progenitor particle which is alive at time $0$,
  sometimes called the `Eve' particle; and each offspring of the particle with label $u \in \tree$
  receives a label $u(m,k,i)$, for some choice of $m,k,i$ which will be explained shortly.
  
 Let $(a,\sigma,\nu)$ be a triple of characteristics
  satisfying the conditions outlined in the introduction,
  and let $\kappa$ be the cumulant given by \eqref{e:kappa}.
  We assume throughout
  that $\nu(\{\mathbf{0}\}) = 0$, where $\mathbf{0} \coloneqq (0,\dotsc)$
  Our results will still hold without this condition,
  but it simplifies notation and proofs by allowing us to ignore the possibility that particles are killed outright.

   Let $(b_n)_{n\ge 0} \subset [0,\infty)$ be a strictly increasing sequence such 
   that $b_0 = 0$ and $b_n \to \infty$; this will be a fixed sequence of truncation
   levels, which will be assumed given throughout this work.
   For $b\ge 0$, we
   let $k_b\from \Pp\to\Pp$ be given by
   \begin{equation}\label{e:kb}
    k_b(p_1,p_2,p_3,\dotsc) = (p_1,p_2\Indic{p_2 > e^{-b}},p_3\Indic{p_3> e^{-b}},\dotsc) , 
    \end{equation}
   and define the \emph{truncated dislocation measure}
   via the pushforward $\nu^{(b)} = \nu \circ k_b^{-1}$.
   
   \skippar
  \label{s:intuitive}%
  We now consider $n\ge 0$ to be fixed; we are going to define the \emph{branching Lévy process
  truncated at level $b_n$}.
  Since the labelling is a little more complex than usual, let us first give an intuitive
  description of this process.
  The process begins at time zero with a single particle having label $\varnothing$, and positioned
  at the origin.
  The spatial position of the particle
  follows a spectrally negative Lévy process $\xi_{\varnothing}$ with
  Laplace exponent $\Psi^{(b_n)}$ defined by
  \[ \Psi^{(b_n)}(q) = \frac{1}{2}\sigma^2q^2 + \mBigl( a + \int_{\Pp\setminus\Pp_1} (1-p_1)\nu^{(b_n)}(\dd \pp) \mr)q
    + \int_{\Pp_1} \mbigl[ p_1^q-1+(1-p_1)q \mr] \, \nu^{(b_n)}(\dd \pp),
  \]
  where
  $\Pp_1$ is the set of sequences with at most one non-zero element,
  \[ \Pp_1 = \{ \pp \in \Pp: p_2 = 0 \}. \]
  Crucially, the moment condition \eqref{e:nu}
  implies that the pushforward
  $\nu^{(b_n)}\rvert_{\Pp_1} \circ \log^{-1}$
  is indeed a L\'evy measure, so $\Psi^{(b_n)}$ is the Laplace exponent
  of a L\'evy process.
  Moreover, $\nu^{(b_n)}$ restricted to $\Pp\setminus \Pp_1$
  is finite.
  
  At time $T_{\varnothing,1}$, having an exponential distribution
  with parameter 
  $\lambda_{b_n} \coloneqq \nu^{(b_n)}(\Pp\setminus\Pp_1)<\infty$, the particle
  $\varnothing$ branches. Take $\pp$ to be a random variable
  with distribution $\nu^{(b_n)}\rvert_{\Pp\setminus\Pp_1}/ \lambda_{b_n}$,
  and scatter particles in locations $\xi_{\varnothing}(T_{\varnothing,1}-) + \log p_i$,
  for $i\ge 1$.
  The particle in location $\xi_{\varnothing}(T_{\varnothing,1}-) + \log p_1$ retains the
  label $\varnothing$.
  Let us momentarily define the \emph{level} of
  a child as the unique natural number $m$ such that
  $e^{-b_{m-1}} \ge p_i > e^{-b_m}$ .
  The particles other than $\varnothing$ receive labels
  $\varnothing(m,1,j) = (m,1,j)$, where $m$ is the level, and
  $j$ is the minimal
  natural number such the initial location of $(m,1,j)$
  in $\RR$ is less than or equal to that of $(m,1,j-1)$
  (recall that particles are scattered downwards.)
  
  After this first branching event, the particle $\varnothing$ continues
  to perform a Lévy process, and then at time $T_{\varnothing,1}+T_{\varnothing,2}$,
  with $T_{\varnothing,2}$ independent of and
  equal in distribution to $T_{\varnothing,1}$, it branches
  again. Particles are scattered according to the same rule.
  This time, a child of level $m$ receives a label of form
  $(m,2,j)$ if children of level $m$ were previously seen (with the $2$
  indicating that this is the second such branching event
  of $\varnothing$ giving rise to particles in class $m$.) If not,
  a child of level $m$ receives a label of form $(m,1,j)$.
  The evolution of $\varnothing$ proceeds in this manner.
  
  Meanwhile, each particle $u$ which was already born has the same evolution.
  It performs a Lévy process $\xi_{u}$ with the same law as $\xi_{\varnothing}$,
  and after waiting a period $T_{u,1}$, independent of and
  equal in distribution to $T_{\varnothing,1}$, it branches.
  Its children are scattered in the same way as before, but they receive labels $u(m,1,j)$;
  and, subsequently, at the $k$-th branching event of $u$ giving rise to children of
  level $m$, these children receive labels $u(m,k,j)$.
  
  A sketch illustrating the labelling scheme appears in \autoref{f:Z}.
  
  \skippar Having established the main idea, we now give a rigorous definition of the
  branching Lévy process truncated at level $b_n$.

  Strictly speaking, all the symbols we define in the next few paragraphs
  should have an annotation of the sort $\cdot^{(b_n)}$, but this would be rather cumbersome.
  The notations $a_\cdot$, $\xi_\cdot$, $N_\cdot$, $T_\cdot$, $\Delta^{(\cdot)}$ and $\mathcal{Q}_{\cdot}$,
  shortly to be defined,
  will not appear again in the sequel, so we warn that they depend implicitly on $n$;
  and all other notations will either receive an annotation or will turn out not to depend on $n$ after all.
  
  \newcommand\ML{\mathsf{ML}}
  \newcommand\Lu{\mathsf{L}}

  Emulating \cite{Ber-cfrag}, we define the following random elements.
  \begin{itemize}
    \item 
      $(\xi_u)_{u\in \tree}$, a family of i.i.d.\ Lévy processes with Laplace exponent $\Psi^{(b_n)}$.
    \item
      $(N_u)_{u\in\tree}$, a family of i.i.d.\ Poisson processes $(N_u(t), t\ge 0)$ with rate 
      $\lambda_{b_n} = \nu^{(b_n)}(\Pp\setminus\Pp_1)$.
      Let us denote by $(T_{u,p})_{p\ge 0}$ the inter-arrival times of $N_u$, so that
      $T_{u,0} = 0$ and $T_{u,p} = \inf\{ t\ge T_{u,p-1} : N_u(t) = N_u(T_{u,p-1})+1 \} - T_{u,p-1}$, for $p\ge 1$.
    \item
      $(\Delta^{(u,p)})_{u\in\tree, p\ge 1}$, a family of i.i.d.\ elements of $\Pp\setminus\Pp_1$ with distribution
      $\nu^{(b_n)}\rvert_{\Pp\setminus\Pp_1}/ \lambda_{b_n}$.
  \end{itemize}
  In the above list, $\xi_u$ represents the motion of the particle with label $u$,
  ignoring the times at which it branches; $N_u$ jumps at the branching times of $u$; and
  the mass-partition
  $\Delta^{(u,p)} = (\Delta^{(u,p)}_i)_{i\ge 1}$
  encodes the relative locations of $u$ and its children at the $p$-th time that $u$ branches. Moreover, these three families are independent one of the
  others. 
  
  Our first step is to divide the $\Delta^{(u,p)}_\cdot$ into (disjoint) classes, which correspond to the
  truncation level of the children they represent.
  We make the definition that:
  \begin{equation}\label{e:L} 
    \Lu(y) \text{ is the unique } m \ge 1 \text{ such that }e^{-b_{m-1}} \ge y > e^{-b_m}.
  \end{equation}
  For $l\ge 1$, let $\Delta^{(u,p,l)} = \mbigl(\Delta^{(u,p)}_j : j\ge 2 \text{ such that } \Lu(\Delta^{(u,p)}_j) = l\mr)^{\downarrow}$,
  where $\cdot^{\downarrow}$ indicates decreasing rearrangement of the sequence.
  For every $l\ge 1$, we regard the finite sequence $\Delta^{(u,p,l)}$ as being an element of $\Pp$, 
  by filling the tail with zeroes. Note that $\Delta^{(u,p,l)} = \mathbf{0}$ for all $l>n$. 
  
  Next, for each label $u$, we will give
  definitions for certain random elements. These are:
  $a_u \in[0,\infty)$, the birth time of $u$;
  $\Zz_u = (\Zz_u(t), t\ge 0)$, with $\Zz_u(t) \in \RR$ representing
  the position of $u$ at time $t$;
  and 
  \[K_u^{(b_n)} = (K_u^{(b_n)}(t), t\ge 0), ~\text{with}~ K_u^{(b_n)}(t) = (K_u^{(b_n)}(t,l) : l\ge 1) \in (\NN\cup \{0\})^{\NN}. \]
  The latter sequence has the interpretation that
  $K_u^{(b_n)}(t,l)$ is the number of branching events which particle $u$ has had up to time
  $t$ in which at least one child with label of the form $u(l,k,i)$, for any $k,i \in \NN$, was born.
  
  For the particle $\varnothing$, let
  \begin{eqnarr*}
    a_\varnothing = 0; 
    \qquad \Zz_\varnothing(t) &=& \xi_\varnothing(t) + \sum_{p=1}^{N_\varnothing(t)} \log\Delta^{(\varnothing,p)}_1, \quad t \ge 0; \\
    K_\varnothing^{(b_n)}(t, l) &=& \sum_{p=1}^{N_\varnothing(t)} \Indic{\Delta^{(\varnothing,p,l)} \ne \mathbf{0}}, \quad t\ge 0,\,l\ge 1.
  \end{eqnarr*}
  
  For the remaining particles, we first need a bit of notation:
  let 
  \[ 
    \mathcal{Q}_{u,m}(k) = \inf \mBigl\{ P \in \NN : \sum_{p=1}^{P} \Indic{\Delta^{(u,p,m)} \ne \mathbf{0}} = k \mr\},
    \qquad k \in \NN,
  \]
  with the convention that $\inf\emptyset = \infty$. Thus, $\mathcal{Q}_{u,m}(k)$ is the number of
  birth events of $u$ which take place until the $k$-th event at which the sequence $\Delta^{(u,p)}$ contains
  at least one element $y$ with $\Lu(y) = m$.
  Fix $u\in\tree$ and $(m,k,i) \in \NN^{3}$ arbitrary, and write $u' = u(m,k,i)$. Then let:
  \begin{eqnarr*}
    a_{u'} &=& a_u + \sum_{p=1}^{\mathcal{Q}_{u,m}(k)} T_{u,p}; \\
    \Zz_{u'}(t) &=& \Zz_{u}(a_{u'}-) + \log \Delta^{(u,\mathcal{Q}_{u,m}(k),m)}_i
    + \xi_{u'}(t-a_{u'}) 
    + \sum_{p=1}^{N_{u'}(t-a_{u'})} \log \Delta^{(u',p)}_1 , \qquad t\ge a_{u'};
    \\
    K_{u'}^{(b_n)}(t, l) &=& \sum_{p=1}^{N_{u'}(t-a_{u'})} \Indic{\Delta^{(u',p,l)} \ne \mathbf{0}}, \qquad t\ge a_{u'},\, l\ge 1.
  \end{eqnarr*}
  We define $a_{u'} = \infty$ if either $\mathcal{Q}_{u,m}(k) = \infty$ or $\Delta^{(u,\mathcal{Q}_{u,m}(k),m)}_i = 0$. 
  In particular, $a_{u'} = \infty$ for every $u' = u(m,k,i)$ with $m>n$.
  
  We are now in a position to define the following elements:
  \begin{defn}
  	Let $n\ge 0$, then
  \[ \Zz^{(b_n)}(t) = \sum_{u\in\tree} \delta_{\Zz_u(t)}\Indic{a_u \le t}, \qquad t\ge 0, \]
	 is the \emph{branching Lévy process truncated at level $b_n$},
  and
  \[ \bar\Zz^{(b_n)}(t) = \sum_{u\in\tree} \delta_{(u,K_u^{(b_n)}(t),\Zz_u(t))}\Indic{a_u\le t}, \qquad t\ge 0, \]
  is the \emph{labelled branching Lévy process truncated at level $b_n$}.
  \end{defn}
  \begin{rem}
  	In the construction above, the role of the component $K^{(b_n)}_u$,
  	which records some information about the children of $u$,
  	is simply
  	to ensure that the process $\bar{\Zz}^{(b_n)}$ is Markov
  	(see the forthcoming \autoref{l:Zbar-bp}.)
  	Without the inclusion of this mark, 
  	if a particle $u$ branches at time $t$, it is not possible to determine
  	the labels of its children solely from $\bar{\Zz}^{(b_n)}(t)$.
  	We emphasise that the unlabelled process, $\Zz^{(b_n)}$, is always Markov
  	\cite[p.~1272]{Ber-cfrag}.
  \end{rem}  
  
  From the labelled branching L\'evy processes, let us also define
  \[ \tree_t^{(b_n)} = \{ u \in \tree : \exists k \in (\NN \cup\{0\})^\NN, z \in\RR \text{ such that } \bar\Zz^{(b_n)}(t)\{(u,k,z)\} = 1 \} , \qquad t\ge 0, \]
  which is the set of labels of particles present at time $t$.
  
  We introduce now the following function, which will be required to
  understand the un-truncated process.
  For $u\in\tree$, define
  \[ \ML(u)
    = \max\{ r \ge 1: \text{there exist } u', k, i, u'' \text{ such that }
    u = u'(r,k,i)u'' \}.
  \]
  Thus, $\ML(u)$ can be seen as the minimal value of $r$ for which a particle
  with label $u$ could appear in
  the construction of $\bar{\Zz}^{(b_r)}$,
  and indeed, if $u \in \tree_t^{(b_n)}$, then $\ML(u) \le n$.
 
  Of these processes, $\Zz^{(b_n)}$ is a branching
  L\'evy process with characteristics
  $(a,\sigma,\nu^{(b_n)})$ in the sense of \citet[Definition 1]{Ber-cfrag},
  and the others are
  our extensions.
  In particular, we have by \cite[Theorem 1]{Ber-cfrag} that
  $\lE\mbigl[ \sum_{u\in\tree_t^{(b_n)}} e^{q\Zz_u(t)}\mr] = e^{t\kappa^{(b_n)}(q)}$, for all $q\in \RR$, 
  where 
     \begin{eqnarr*}
     	\kappa^{(b_n)}(q) &\coloneqq& \frac{1}{2}\sigma^2q^2 + aq + \int_{\Pp} \mBigl[ \sum_{i\ge 1} p_i^q - 1 + (1-p_1)q \mr] \, \nu^{(b_n)}(\dd \pp)
     	\\
     	&=& \frac{1}{2}\sigma^2q^2 + aq + \int_{\Pp} \mBigl[ p_1^q-1+(1-p_1)q + \sum_{i\ge 2} p_i^q\Indic{p_i>e^{-b_n}} \mr] \, \nu(\dd \pp), \qquad q \in \RR.
     \end{eqnarr*}
      This function represents the cumulant of the truncated branching Lévy process.  
      
  \begin{figure}[t]
      \centering
      \includegraphics[width=0.8\textwidth]{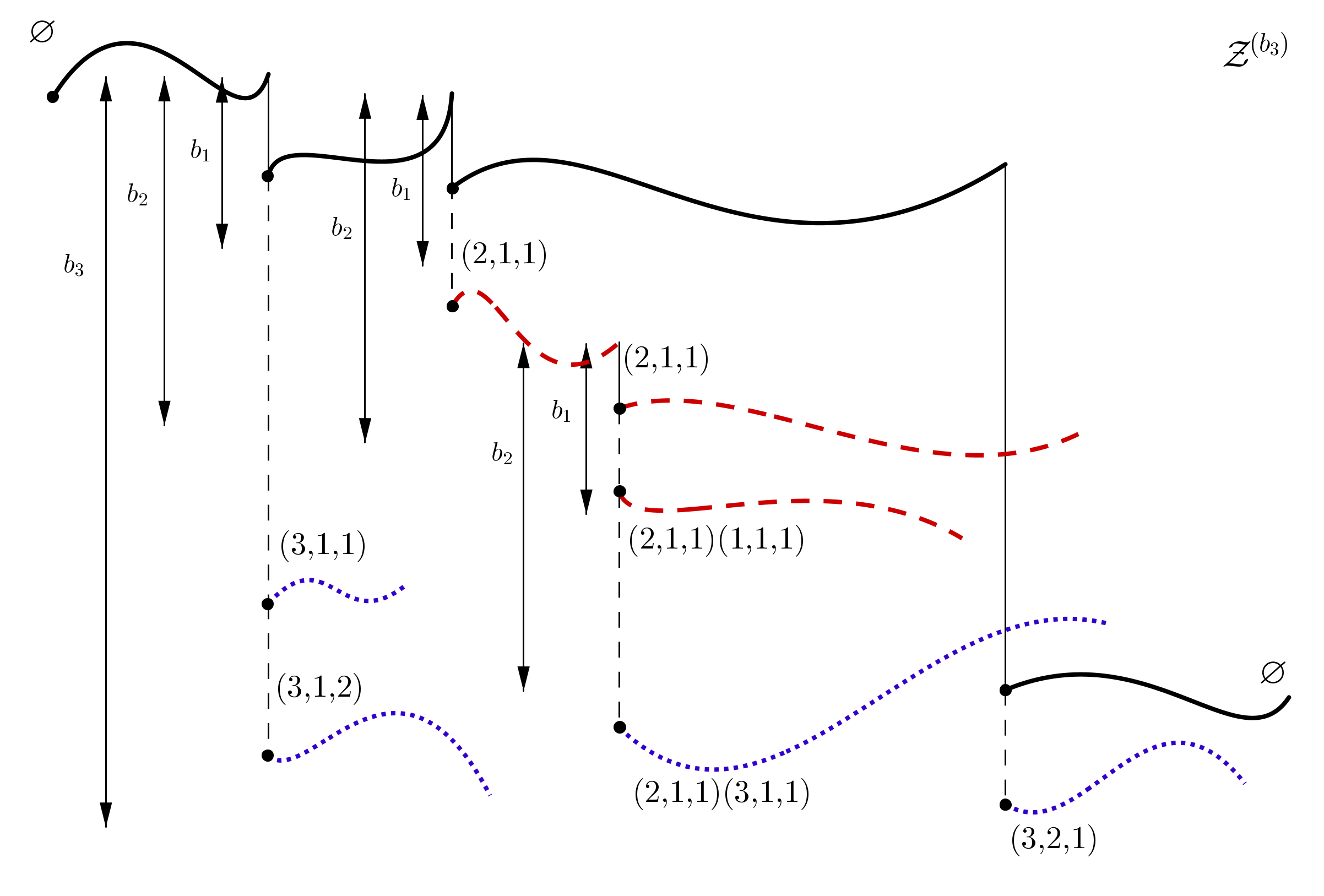}
      \caption{A sketch of the construction and labels of a (truncated)
      branching L\'evy process,
      with truncation levels marked at certain birth events.
      The path in solid black represents the process $\Zz^{(b_1)}$, which in this
      particular instance includes only the Eve particle $\varnothing$. The paths
      in dashed red represent the particles in the process
      $\Zz^{(b_2)}\setminus\Zz^{(b_1)}$; note that these are precisely the particles
      $u$ for which $\ML(u) = 2$.
      The paths in dotted blue represent the particles in the process
      $\Zz^{(b_3)}\setminus\Zz^{(b_2)}$,
      which are those particles $u$ such that $\ML(u) = 3$.
      } 
      \label{f:Z}
  \end{figure}

  \skippar
  Having defined the truncated branching Lévy process $\bar{\Zz}^{(b_n)}$,
  we introduce the idea of further truncating it at level $b_m\le b_n$. That is,
  we consider keeping, at each branching event, the child particle which is the closest to
  the parent, and suppressing the other children if and only if their distance to the position of the parent
  prior to branching is larger than or equal to $b_m$, together with their descendants. 
  Mathematically, for $m\le n$, we let

  \[ (\tree_t^{(b_n)})^{(b_m)} = \mbigl\{u\in \tree_t^{(b_n)}~:~ \ML(u) \le m\mr\} , \quad t\ge 0.\]
  We then define 
  \begin{equation}
    (\Zz^{(b_n)})^{(b_m)}(t) = \sum_{u \in(\tree_t^{(b_n)})^{(b_m)}} \delta_{\Zz_u(t)} , \quad t\ge 0,
    \label{e:bnbm}
  \end{equation}
  which is the \emph{truncation of $\Zz^{(b_n)}$ to level $b_m$}, 
  and similarly
  \begin{eqnarr*}
    (K_u^{(b_n)})^{(b_m)}(t,l) &=& K_u^{(b_n)}(t,l)\Indic{l\le m}, \quad t\ge 0,\, l\ge 1;
    \\
    (\bar\Zz^{(b_n)})^{(b_m)}(t) &=& \sum_{u\in(\tree_t^{(b_n)})^{(b_m)}}\delta_{(u,(K_u^{(b_n)})^{(b_m)}(t),\Zz_u(t))}, \quad t\ge 0.
  \end{eqnarr*}

  With this definition, we get the following lemma.
  
  \begin{lem}\relax\label{l:consistency}
    Let $m\le n$. Then $(\Zz^{(b_n)})^{(b_m)}$ is equal in law to $\Zz^{(b_m)}$
    and $(\bar\Zz^{(b_n)})^{(b_m)}$ is equal in law to $\bar\Zz^{(b_m)}$.
  \end{lem}
  \begin{proof}
    The first statement is \cite[Lemma 3]{Ber-cfrag}, and the second follows by considering
    the intuitive description of the labels beginning on page~\pageref{s:intuitive}:
    if all $u$ with $\ML(u) > m$ are removed,
    then those elements do not appear in $\bar\Zz$, and the sequence $(K_{u'}^{(b_n)}(t))^{(b_m)}$
    for the remaining $u'$ simply erases the record of birth events that would have given rise to those
    erased $u$.
  \end{proof}
  
  We therefore see that both the labels and the positions of the particles are consistent
  under truncation, as are the marks $K_u^{(b_\cdot)}$. By the Kolmogorov extension theorem,
  we can construct, simultaneously on the same probability space, 
  a collection of processes $(\Zz^{(b_n)})_{n\ge 0}$ and $(\bar\Zz^{(b_n)})_{n\ge 0}$
  with the property that the equality in law of \autoref{l:consistency}
  is replaced by almost sure equality.
  
  \skippar
  Thus, we are able to define the following (un-truncated) processes:
  \begin{defn}
  	The \emph{branching Lévy process with characteristics $(a,\sigma,\nu)$} is 
  	\[ \qquad \Zz(t) = \lim_{n\ge 0} \Zz^{(b_n)}(t), \qquad t\ge 0,\]
   where here, for each $t\ge 0$, $(\Zz^{(b_n)}(t))_{n\ge 0}$ is an increasing sequence of 
   point measures (i.e., $\Zz^{(b_n)}(t) (A)\le \Zz^{(b_m)}(t)(A)$ for every $n\le m$ and any Borel set $A\in \mathcal{B}(\RR)$) and we take $\Zz(t)$ as their increasing limit.  
  \end{defn}
  	  For the (un-truncated) process $\Zz$, the set of labels of particles present up to time $t\ge 0$ is  
  	  \[ \tree_t = \bigcup_{n\ge 0} \tree_t^{(b_n)}, \qquad t\ge 0. \] 
  	  Furthermore, we define $K_u$ by using a diagonal sequence: for each $t\ge 0$ and $u\in\tree_t$, 
  	  \[
  	  K_u(t) = (K_u(t,l))_{l\ge 1}, ~\text{where}~ K_u(t,l) = K_u^{(b_l)}(t,l). 
  	  \]
  \begin{defn}
  	 		The \emph{labelled branching Lévy process with characteristics $(a,\sigma,\nu)$} is 
  	 		\[ \bar\Zz(t) = \sum_{u\in\tree_t} \delta_{(u,K_u(t),\Zz_u(t))}, \qquad t\ge 0. \]
  	\end{defn}
  
  In particular, since $\kappa^{(b_n)}(q) \upto \kappa(q)$ whenever $q\in\dom\kappa$,
  we have that \[ \lE\mBigl[ \sum_{u\in\tree_t} e^{q\Zz_u(t)}\mr] = e^{t\kappa(q)}, \qquad q\in\dom\kappa, \] which
  is an important property of the process.
  
  \begin{rem}\fakephantomsection \label{r:misc}
    \begin{enumerate}
      \item
        In \cite{BM-bLp,BBCK-maps}, growth-fragmentations
        are studied in which 
        upward jumps of the particle locations (with or without associated branching)
        are permitted, under some non-exploding assumption. This can be accommodated in our construction
        as well, simply by removing
        the restriction that the motion processes $(\xi_u)_{u\in \mathcal{U}}$ be spectrally negative
        (and, if necessary, incorporating branching at upward jumps)
        thereby giving versions of these processes with labels and genealogies.
      \item
        The processes $\Zz^{(b_n)}$ and $\bar\Zz^{(b_n)}$ are (labelled) branching
        L\'evy processes in their own right, having characteristics
        $(a,\sigma,\nu^{(b_n)})$.
      \item\label{i:misc:3}
        We wish to emphasise that, despite the technical appearance of our label
        definitions, they can be found deterministically once the unlabelled
        branching L\'evy process
        is known. In particular, if we have all $\Zz^{(b_n)}$ defined on the same
        probability space, and we are given a single sample from this space,
        then a sample of the process $\bar{\Zz}^{(b_n)}$ can be constructed, without
        extra randomness, using the intuitive
        definition of the labels on page \pageref{s:intuitive}. 
        This will be important in \autoref{s:forward}.
    \end{enumerate}
  \end{rem}
  
  \subsection{Regularity and the branching property}\label{s:branching}
  
  
  One of the key results of \cite{Ber-cfrag} was the branching property of the compensated
  fragmentation $\Zb$. This result extends naturally to $\Zz$, and we shall shortly
  give an explicit statement of it for $\bar\Zz$. However, we first elaborate
  a little on the state space of $\bar\Zz$, and consider the regularity of the process.
  
  We first expand on the space $\tree$. Some of the definitions here will not be
  needed until the next section, but we give them here for ease of reference.
  We define relations
  $\preceq$ and $\prec$ on $\tree$ to denote ancestry, 
  so $u\preceq v$ if there exists some $u' \in \tree$ such that $v = uu'$,
  and $u\prec v$ if $u\preceq v$ and $u\ne v$.
  Using this, we define ancestors and descendants
  as follows, which is a little subtle due to immortality of particles.
  If $s<t$ and $v\in\tree_t$, we define $u = \Anc(s;v)$ to be the
  largest (with respect to $\preceq$) element of $\tree_s$ such that $u\preceq v$.
  Conversely, for $u\in\tree_s$, we define $\Desc(s,u;t) = \{ v \in \tree_t : u = \Anc(s;v)\}$.
  We also define $\abs{u}$ to be the unique $n \in \NN\cup\{0\}$
  such that $u \in (\NN^3)^n$,
  that is, the generation of $u$;
  and $(u_i)_{1\le i\le n}$ to be those elements of $\NN^3$ such that $u = u_1\dotsb u_n$.
  We extend this so that $u_i = (0,0,0)$ if $\abs{u} < i$.
  Finally, we consider $\tree$ be endowed with the metric
  $\rho(u,v) = \sum_{i\ge 1} \norm{u_i - v_i}$, where here $\norm{\,\cdot\,}$ is the
  usual Euclidean norm on $\RR^3$.

  Define the space $\mathcal{L}$ to consist of those sequences $K = (K(l))_{l\ge 1}$ in the set $(\NN\cup\{0\})^{\NN}$
  for which the function $\norm{K}_{\mathcal{L}} = \sum_{l\ge 1} \exp(-l-e^{4b_l}) \abs{K(l)}$
  is finite; then $(\mathcal{L},\norm{\,\cdot\,}_{\mathcal{L}})$ is a normed vector
  space which is isomorphic to $\ell^1$,
  via the map $\phi\from \mathcal{L} \to \ell^1$, $\phi(K) = (e^{-l-e^{4b_l}}K(l) : l\ge 1)$.

  Now let $\mathbf{X} = \tree\times \mathcal{L} \times \RR$.
  This is a complete, separable metric space when given
  the product metric $d\bigl((u,K,x),(u',K',x')\bigr) = \rho(u,u')+\lVert K-K'\rVert_{\mathcal{L}} + \lvert x-x'\rvert$.
  It will prove useful to define
  $\mathcal{M}_p(\mathbf{X})$ to be the set of point measures on $\mathbf{X}$
  which are finite on bounded subsets of $\mathbf{X}$. We give this a metric as follows
  (see \cite[\S A2.6]{DV-pp1}). Let $q \in (\dom \kappa)^\circ$ be chosen arbitrarily,
  and let $x_0 = (\varnothing,\mathbf{0},0) \in \mathbf{X}$.
  If $\mu,\mu'$ are point measures on $\mathbf{X}$, let
  \[ d_q(\mu,\mu') = \int_0^\infty q e^{-qr} \frac{d^{(r)}(\mu_r,\mu_r')}{1+d^{(r)}(\mu_r,\mu_r')} \, \dd r , \]
  where
  $\mu_r = \mu\rvert_{B_r(x_0)}$ is the measure $\mu$ restricted to
  the open ball $B_r(x_0)$ of radius $r \ge 0$ around $x_0$, and
  $d^{(r)}$ is the L\'evy--Prokhorov metric
  on  $B_r(x_0)$; this is defined as:
  \begin{multline*} d^{(r)}(\mu,\mu') = \inf\{ \epsilon \ge 0 : \text{for all } F \subset B_r(x_0) \text{ closed, } \\
  \mu_r(F) \le \mu_r'(F^{\epsilon})+\epsilon \text{ and } \mu_r'(F) \le \mu_r(F^{\epsilon})+\epsilon \}, \end{multline*}
  where $F^\epsilon \coloneqq \{x \in B_r(x_0) : \text{there exists } y\in F \text{ such that } d(x,y) < \epsilon\}$.

  For a labelled branching L\'evy process $\bar\Zz(t) = \sum_{u\in\tree_t} \delta_{(u,K_u(t),\Zz_u(t))}$, one may show that for any $u\in \tree$ and $t\ge 0$, $\norm{K_u(t)}_{\mathcal{L}}<\infty$
  almost surely. 
  Therefore, we may regard $\bar\Zz$ as taking values in
  the complete separable metric space $\mathcal{M}_p(\mathbf{X})$
  with metric $d_q$. Furthermore,
  we have the following pair of results:
  \begin{lem}\label{l:cv-Zbn}
    For $q \in (\dom\kappa)^\circ$ and $t\ge 0$, $\sup_{s\le t} d_q( \bar\Zz(s),\bar\Zz^{(b_n)}(s)) \to 0$ in probability
    as $n\to\infty$.
    \begin{proof}
      Fix $q \in (\dom \kappa)^\circ$ and $t\ge 0$.
      To begin with,
      \begin{multline}\label{e:cadlag-cv-inter}
        d_q(\bar\Zz(s),\bar\Zz^{(b_n)}(s))
        \le
        d_q\mBigl( \sum_{u \in \tree_s} \delta_{(u,K_u(s),\Zz_u(s))},
        \sum_{u\in\tree_s^{(b_n)}} \delta_{(u,K_u(s),\Zz_u(s))} \mr)
        \\
        {} +
        d_q\mBigl( \sum_{u \in \tree_s^{(b_n)}} \delta_{(u,K_u(s),\Zz_u(s))},
        \sum_{u\in\tree_s^{(b_n)}} \delta_{(u,K_u^{(b_n)}(s),\Zz_u(s))} \mr).
      \end{multline}
      We study the two terms on the right-hand side separately.
      
      We first look at the second term.
      Take $\epsilon = \max\{ \norm{K_u^{(b_n)}(s) - K_u(s)}_{\mathcal{L}} : u \in \tree_s^{(b_n)} \}$; then
      the distance
      between any pair of points
      $(u,K_u(s),\Zz_u(s))$ and $(u,K_u^{(b_n)}(s),\Zz(s))$,
      with $u \in \tree_s^{(b_n)}$,
      is at most $\epsilon$, and so for any closed $F \subset \mathbf{X}$,
      $(u,K_u(s),\Zz_u(s)) \in F$ implies $(u,K_u^{(b_n)}(s),\Zz(s)) \in F^{\epsilon}$,
      and vice versa.
      By the definition of the L\'evy--Prokhorov
      metric,
      we see that for every $r \ge 0$,
      \[
        d^{(r)}\mBigl( \sum_{u \in \tree_s^{(b_n)}} \delta_{(u,K_u(s),\Zz_u(s))},
                \sum_{u\in\tree_s^{(b_n)}} \delta_{(u,K_u^{(b_n)}(s),\Zz_u(s))} \mr)
        \le \max\{ \norm{K_u^{(b_n)}(s) - K_u(s)}_{\mathcal{L}} : u \in \tree_s^{(b_n)} \}.
      \]
      Note that $K_u^{(b_n)}(s,l) = K_u(s,l)$ for all $l\le n$, we have 
      \[
      \norm{K_u^{(b_n)}(s) - K_u(s)}_{\mathcal{L}}
      = \sum_{l\ge n+1} \exp(-(l+e^{4b_l})) K_u(s,l). 
      \]
      The crude estimate $\sup_{s\le t} K_u(s,l) = K_u(t,l) \le \card\Zz^{(b_l)}(t)$,
      the latter being the number of particles in $\Zz^{(b_l)}(t)$, then yields      
      	        \begin{eqnarr*}
      	        	\lE\mbigl[\sup_{s\le t}\max\{ \norm{K_u^{(b_n)}(s) - K_u(s)}_{\mathcal{L}} : u \in \tree_s^{(b_n)} \}\mr]
      	        	&\le& \sum_{l\ge n+1} \exp(-(l+e^{4b_l})) \lE[\card \Zz^{(b_l)}(t)]
      	        	\\
      	        	&=& \sum_{l\ge n+1} \exp\mbigl(\kappa^{(b_l)}(0) t -l-e^{4b_l}\mr) .
      	        \end{eqnarr*}   
      	        Noticing that
      	        \[e^{-2 b_l}\kappa^{(b_l)}(0) = \int_{\Pp} \sum_{i\ge 2}e^{-2 b_l} \Indic{p_i> e^{-b_l}}\, \nu (\dd \pp) \le \int_{\Pp} \sum_{i\ge 2}p_i^2\,  \nu (\dd \pp) =: C<\infty, \] 
      	        we have
      	        \[\sum_{l\ge n+1} \exp\mbigl(t\kappa^{(b_l)}(0) -l-e^{4b_l}\mr)  \le \sum_{l\ge n+1} e^{-l} \exp\mbigl(e^{2b_l}(Ct - e^{2b_l})\mr).	\]	
      For fixed $t\ge 0$,
      the right-hand side tends to zero as $n\to\infty$.
      This ensures that the second term of \eqref{e:cadlag-cv-inter} converges to
      zero in probability.

      Turning to the first term in \eqref{e:cadlag-cv-inter}, we have
      \[ 
      d^{(r)}\mBigl( \sum_{u \in \tree_s} \delta_{(u,K_u(s),\Zz_u(s))},
              \sum_{u\in\tree_s^{(b_n)}} \delta_{(u,K_u(s),\Zz_u(s))} \mr)
      \le \sum_{u \in \tree_s\setminus \tree_s^{(b_n)}} \Indic{\Zz_u(s) \in (-r,r)} . \]
      We now integrate in order to study the $d_q$-distance, and use the
      bound $\Indic{\Zz_u(s) \in (-r,r)} \le e^{q'(r+\Zz_u(s))}$, where
      $q' \in \dom \kappa$ is chosen arbitrarily such that $q' < q$ holds:
      \begin{eqnarr*}
        d_q\mBigl( \sum_{u \in \tree_s} \delta_{(u,K_u(s),\Zz_u(s))},
        \sum_{u\in\tree_s^{(b_n)}} \delta_{(u,K_u(s),\Zz_u(s))} \mr)
        &\le& \int_0^\infty qe^{-qr}
        \sum_{u \in \tree_s\setminus \tree_s^{(b_n)}} \Indic{\Zz_u(s) \in (-r,r)} \, \dd r
        \\
        &\le& \int_0^\infty qe^{(q'-q)r}
        \sum_{u \in \tree_s\setminus \tree_s^{(b_n)}} e^{q' \Zz_u(s)} \, \dd r \\
        &=& \frac{q}{q-q'} \mBigl(
        \sum_{u \in \tree_s} e^{q' \Zz_u(s)} 
        - \sum_{u \in \tree_s^{(b_n)}} e^{q' \Zz_u(s)} 
        \mr).\IEEEyesnumber \label{e:trick}
      \end{eqnarr*}

      Now, the proof is completed using Doob's maximal inequality exactly
      as in \cite[Proof of Lemma 4]{Ber-cfrag}.
    \end{proof}
  \end{lem}

  \begin{cor}[regularity of $\bar\Zz$]\label{c:cadlag}
    The process $\bar\Zz$ possesses a c\`adl\`ag version in $\mathcal{M}_p(\mathbf{X})$.
    \begin{proof}
      This follows from the above lemma exactly as in \cite[Proposition 2]{Ber-cfrag}.
    \end{proof}
  \end{cor}

  Thanks to this result, we can consider $\lP$ to be defined on the space 
  $\Omega = D([0,\infty),\mathcal{M}_p(\mathbf{X}))$ of
  càdlàg functions from $[0,\infty)$ to $\mathcal{M}_p(\mathbf{X})$,
  endowed with the Skorokhod topology; we refer the reader to
  \cite{Bil-conv} for more details on this space.

  The process $\bar\Zz$ has the Markov property, which in this context is usually called 
  the \emph{branching property} and which we now explain.
  We first define translation operators for $u \in \tree$ and $t\ge 0$, as follows.
  Let $\theta_{u,t}\from\Omega \to \Omega$ be such that,
  if
  \[ \bar\Zz(s+t) = \sum_{uu'\in\tree_{s+t}} \delta_{(uu', K_{uu'}(s+t),\Zz_{uu'}(s+t))}
    + \sum_{\substack{u''\in\tree_{s+t}, \\u \not\preceq u''}} \delta_{(u'',K_{u''}(s+t),\Zz_{u''}(s+t))}, \qquad s \ge 0 , \]
  then
  \[ \bar\Zz(s)\circ\theta_{u,t} = \sum_{uu'\in \tree_{s+t}} \delta_{(u',K_{uu'}(s+t)-K_{uu'}(t),\Zz_{uu'}(s+t)-\Zz_{uu'}(t))}, \qquad s \ge 0. \]
    That is, $\theta_{u,t}$ shifts the particle process such that one only observes the particle with label $u$
    and its descendants which are born strictly after $t$; and the particle represented by $u$ is shifted to start
    at the origin, at time $0$, with label $\varnothing$ and no recollection of its genealogical history.
    
  Let $(\FF_t)_{t\ge 0}$ be the natural filtration of
  $\bar\Zz$, namely $\FF_t = \sigma(\bar\Zz(s), s\le t)$,
  and define $\FF_\infty = \sigma\bigl(\cup_{t\ge 0} \FF_t\bigr)$.     
  We then have the following simple result.
  \begin{lem}[branching property]\relax\label{l:Zbar-bp}
   For each $u \in \tree$, let $F_u$ be a bounded, measurable functional. Then
  \[
    \lE\mBigl[ \prod_{u\in \tree_t} F_u( \bar\Zz(s) \circ \theta_{u,t}, s \ge 0) \mm\vert \FF_t \mr]
    = \prod_{u \in \tree_t} \lE[ F_u(\bar \Zz(s), s \ge 0) ] .
  \]
  \end{lem}
  \begin{proof}
    This follows directly from the branching property of $\Zz$ in \cite[p.~1272]{Ber-cfrag}
    and the construction of the labels.
  \end{proof}

  We remark that, as a consequence of \autoref{c:cadlag}, the constant time $t$
  in the above lemma may be replaced by any $(\FF_t)$-stopping time,
  or indeed by a stopping line in the sense of \cite[\S 4]{BHK-FKPP}.

  \section{Change of measure and backward selection of the spine}
  \label{s:backward}

  For $\omega \in \dom \kappa$, we define the \emph{exponential additive martingale} $\mg{\omega}{\cdot}$
  just as we did in the introduction:
  \begin{equation*}
    \mg{\omega}{t} = e^{-t\kappa(\omega)} \sum_{u\in \tree_t} e^{\omega \Zz_u(t)} , \wh t \ge 0.
  \end{equation*}
  It has been proved in \cite[Corollary 3]{Ber-cfrag} that this is a 
  martingale with unit mean.
  As such, we may make a martingale change of measure, as follows. We define a measure
  $\lQ$ on $\FF_\infty$ by setting, for $A\in \FF_t$,
  \begin{equation}\label{e:lQ-W} \lQ(A) = \lE[ \Ind_A W(\omega,t) ] . \end{equation}
  The martingale property of $W(\omega,\cdot)$ ensures
  that this change of measure is consistent across different choices of $t$, and
  also implies that the process $\bar\Zz$ under $\lQ$ remains a Markov process.
  $\lQ$ is often referred to as an `exponential tilting' of the probability measure $\lP$.
  
  \medskip\noindent  
  Under this tilted measure, we isolate a single particle as the `spine'.
  We first expand the basic probability space $\Omega$ to
  produce $\hat\Omega = \Omega\times\tree^{[0,\infty)}$,
  where $\tree^{[0,\infty)}$ is the set of functions
  from $[0,\infty)$ to $\tree$.
  Introduce for each $t\ge 0$ a random variable $U_t$ such that,
  for $A\subset\Omega$ measurable,
  $A\times\{U_t = u\} = A \times \{ g \in \tree^{[0,\infty)} : g(t) = u\}$.
  Let $\hat\FF_t = \sigma(\FF_t; U_s, s\le t)$ and $\hat\FF_\infty = \sigma\bigl(\cup_{t\ge 0} \hat\FF_t\bigr)$.
  
  We may then extend the definition of $\lQ$ to sets in $\hat\FF_\infty$.
  For $A \in \FF_t$ and $u\in \tree$, let
  \begin{equation}\label{e:lQ}
    \lQ(A; U_t=u) = e^{-t\kappa(\omega)} \lE[ \Ind_A e^{\omega \Zz_u(t)}].
  \end{equation}
  It is well-known 
  (see, for instance, \cite[Theorem 4.2]{HH-spine})
  that events $\hat A \in \hat \FF_t$ may be written as
  $\hat A = \bigcup_{u\in\tree} (A_u \cap \{ U_t = u \})$, with $A_u \in \FF_t$, and so \eqref{e:lQ}
  is equivalent to defining
  \begin{equation}
    \lQ(\hat A) = e^{-t\kappa(\omega)} \lE \mBigl[\sum_{u\in \tree_t} \Ind_{A_u} e^{\omega \Zz_u(t)} \mr] , \qquad \hat A \in \hat{\FF}_t.
    \label{e:lQ'}
  \end{equation}
  The measure $\lQ$ is well-defined,
  in that, if $\hat{A}\in\FF_t$, then the right-hand side of \eqref{e:lQ'}
  reduces simply to \eqref{e:lQ-W}.
  However, in terms of the definition on $\hat{\FF}_\infty$,
  $\lQ$ distinguishes the label $U_t$
  at time $t$, and we call this the \emph{spine label}.
  
  For each fixed $t\ge 0$, if we define $U_t$ via \eqref{e:lQ},
  we can project it backward by setting $U_s = \Anc(s;U_t)$
  for $s\le t$.
  Due to the branching property of $\Zz$, this is
  consistent with evaluating $\lQ$ on $\hat\FF_s$,
  as is made precise in the following lemma.
  \begin{lem}[consistency of $\lQ$]\relax\label{l:lQ-consistent}
    Let $s< t$ and $u \in \tree$. Let $\lQ^t$ indicate the measure $\lQ$
    defined on $\hat\FF_t$ by means of \eqref{e:lQ} and back-projection of $U_t$, and $\lQ^s$ similarly
    for $\lQ$ defined on $\hat\FF_s$.
    If $A \in \FF_s$, then
    \[ \lQ^s(A; U_s = u) = \lQ^t(A; U_s = u) . \]
  \end{lem}
  \begin{proof}
    Firstly, we have
    \[ 
      e^{-t\kappa(\omega)}\lE\mBigl[ \sum_{v \in \Desc(s,u;t)} e^{\omega \Zz_v(t)} \mgiven \FF_s \mr]
      = e^{-s\kappa(\omega)} e^{\omega\Zz_u(s)},
    \]
    due to the branching property. Then, 
    \begin{eqnarr*}
      \lQ^s( A; U_s=u)
      &=& e^{-s\kappa(\omega)}\lE[ \Ind_A e^{\omega\Zz_u(s)}] 
      \\
      &=& e^{-t\kappa(\omega)}
      \lE\mbiggl\{ 
      \Ind_A 
      \lE\mBigl[ \sum_{v \in \Desc(s,u;t)} e^{\omega\Zz_v(t)} \mgiven \FF_s \mr] \mr\} \\
      &=& e^{-t\kappa(\omega)}
      \lE\mBigl[ \Ind_A \sum_{v \in \Desc(s,u;t)} e^{\omega\Zz_v(t)} \mr]
      \\
      &=& \mathbb{E}_{\lQ^t} \mBigl[\Ind_A  \sum_{v \in \Desc(s,u;t)} \Indic{U_t=v} \mr]
      = \lQ^t( A; U_s=u). \qedhere
    \end{eqnarr*}
  \end{proof}
  
  \skippar
  We refer to the process $(\bar\Zz,U) = ( (\bar\Zz(t),U_t), t\ge 0)$ as the
  \emph{branching L\'evy process with spine}.
  In order for it to be useful, it is important that
  $(\bar\Zz,U)$ should retain the branching property. 
  For the sake of clarity, we keep the time-annotation
  $\lQ^t$ which was introduced in the last lemma.
  \begin{lem}[branching property of $(\bar \Zz,U)$]\relax\label{l:Zbar-U-bp}
    Fix $t\ge s\ge 0$. Let $F_v$ be an $\FF_{t-s}$-measurable functional for each $v\in \tree$, and
    let $G$ be $\sigma(U_{t-s})$-measurable. Then,
    \[ \lQ^t \mbiggl[ (G\circ \theta_{U_s,s}) \cdot \prod_{v\in\tree_s} (F_v\circ\theta_{v,s}) \mm\vert \hat\FF_s \mr]
     = \mbigl. \mBigl( \prod_{\substack{v\in \tree_s \\ v\ne U_s}} \lP[F_v] \mr)
      \cdot \lQ^{t-s}[G\cdot F_u]\mr\rvert_{u=U_s} .
    \]
  \end{lem}
  \begin{proof}
    By Kolmogorov's definition of conditional expectation and the definition of $\hat\FF_s$,
    it is sufficient to prove that, for $K$ an $\FF_s$-measurable functional and $u\in\tree$,
    \begin{equation}\label{e:Zbar-U-bp-inter}
      \lQ^t\mbiggl[ K \Indic{U_s = u} (G\circ \theta_{u,s}) \cdot \prod_{v\in\tree_s} (F_v \circ \theta_{v,s}) \mr]
      = \lQ^t\mbiggl[ K \Indic{U_s=u}
     \mBigl( \prod_{\substack{v\in\tree_s \\v\ne u}} \lP[F_v] \mr)
      \cdot \lQ^{t-s}[G\cdot F_u] \mr].
    \end{equation}
    Fixing $G = \Indic{U_{t-s}=u'}$, for some $u'\in\tree$, the left-hand side is equal to
    \begin{eqnarr*}
      \eqnarrLHS{
        e^{-t\kappa(\omega)}
        \lP\mbiggl[ K \Indic{u=\Anc(s;uu')} e^{\omega  \Zz_{uu'}(t)}
         \prod_{v\in \tree_s} (F_v\circ\theta_{v,s}) \mr]
      }
      \qquad \qquad &=&
      e^{-t\kappa(\omega)} \lP\mbiggl[ K e^{\omega \Zz_u(s)}
      \mBigl( \prod_{\substack{v\in\tree_s \\ v \ne u}} \lP[F_v] \mr)
      \cdot \lP[ F_u e^{\omega \Zz_{u'}(t-s)} \Indic{u'\in \tree_{t-s}} ] \mr]
      \\
      &=&
      e^{-s\kappa(\omega)}\lP\mbiggl[ K e^{\omega\Zz_u(s)} 
     \mBigl( \prod_{\substack{v \in \tree_s\\ v\ne u}} \lP[F_v] \mr)
      \cdot
      \lQ^{t-s}[ F_u \Indic{U_{t-s} = u'}] \mr] 
      \\
      &=&
      \lQ^s \mbiggl[ K \Indic{U_s=u} 
      \mBigl( \prod_{\substack{v \in \tree_s\\ v\ne u}} \lP[F_v] \mr)
      \cdot \lQ^{t-s}[ G\cdot F_u] \mr]
    \end{eqnarr*}
    where in the second line we have used \autoref{l:Zbar-bp} and the fact that the event $u=\Anc(s;uu')$ is equivalent
    to the event that $uu'$ is born after time $s$ (or $u'=\varnothing$);
    and in the third and fourth lines we have used the definition of $\lQ^\cdot$. An 
    appeal to \autoref{l:lQ-consistent} yields \eqref{e:Zbar-U-bp-inter}, which completes the proof.
  \end{proof}

  \skippar
  From now on we will drop the time-annotations $\lQ^t$ and simply use the notation $\lQ$.
  Our primary goal in the remainder of the article is to characterise the law
  of the process $(\bar\Zz,U)$ in terms of well-understood objects.

  \section{Forward construction of the process with spine}
  \label{s:forward}
  
  In this section, we give a construction of a Markov process with values
  in the set of point measures and with a certain distinguished line of
  descent. The process, which we will write as $(\bar\Yy,V)$,
  is regarded as being defined under $\lQ$,
  and we call it the
  \emph{decorated spine process} with parameters $(a,\sigma,\nu,\omega)$.
  In the next section, we will
  show that it coincides in law with the process $(\bar{\Zz},U)$
  described in \autoref{s:backward}.

  We start with a candidate for the motion of the spine particle
  itself.
  Let $\xi$ be a spectrally negative L\'evy process
  whose
  Laplace exponent has the L\'evy--Khintchine representation
  \[ \Ess_\omega\kappa(q) \coloneqq \kappa(q+\omega)-\kappa(\omega) = 
    \frac{1}{2}\sigma^2q^2 + a_{\omega}q + \int_{(0,1)} (y^q-1-ql(y)) \, \pi(\dd y) , \qquad q \ge 0, \]
  where
  \begin{eqnarr*}
    \pi(\dd y) &=& \sum_{i\ge 1} y^\omega \nu(p_i \in \dd y) 
    \text{, that is, } \int_{(0,1)} f(y) \, \pi(\dd y) = \int_{\Pp} \sum_{i\ge 1} p_i^\omega f(p_i) \, \nu(\dd \pp), \\
    l(y) &=& \Indic{\abs{\log y} \le 1}\log y, \\
    a_\omega &=& a+\omega\sigma^2 + \int_{\Pp} \mBigl[ 1 - p_1 + \sum_{i\ge 1} p_i^\omega l(p_i) \mr] \,\nu(\dd \pp) .
  \end{eqnarr*}
  Note that
  in particular, the Lévy measure of $\xi$ is given by
  the pushforward $\Pi \coloneqq \pi \circ \log^{-1}$.
  
  The motivation for this definition of $\xi$ is that, if
  $\nu(\Pp\setminus\Pp_1) <\infty$,
  then by \cite[Proposition 3.4]{BerSte}
  the process $(\Zz_{U_t}(t), t \ge 0)$ under $\lQ$
  is known to be equal in law to the
  process $\xi$; this is not difficult to prove
  even in the absence of said finiteness condition,
  but it will be a corollary of the main theorem in the next
  section, so we do not pursue this here.
  
  \skippar
  We regard $\xi$ as representing the position of the spine particle,
  and our goal is now to construct the rest of the branching L\'evy process
  around it. There will be
  three steps to this: firstly, we take the Poisson random measure
  giving the jump times and sizes
  of $\xi$. We then add decorations to this which indicate the additional offspring
  which should be present due to the branching structure; and in the final step,
  we graft independent branching L\'evy processes (under $\lP$) onto this structure.

  \skippar
  Next we require a short lemma establishing the existence of a conditional measure.

  \begin{lem}\relax\label{l:cond}
    For each $i \in \NN$,
    there exists a probability kernel $\nu_i$ from $(0,1)$ to $\Pp$ 
    which associates to each $y\in (0,1)$ a probability measure $\nu_i(\dd \pp \mid y)$, such that
    \[ \int_{\Pp} f(\pp) \, \nu(\dd\pp) = \int_{\Pp\times (0,1)} f(p_1,\dotsc,p_{i-1},y,p_{i+1},\dotsc) \, \nu_i(\dd \pp \mid y) \nu(p_i \in \dd y). \]
    \begin{proof}
    { \let\rho\varrho
    Let $\rho_i \from \Pp \to (0,1)$ be given by $\rho_i(p_1,p_2,\dotsc) = p_i$, and let $\nu_i = \nu \circ \rho_i^{-1}$.
    We seek measures $\nu_i(\cdot \given y)$, such that each $\nu_i(\cdot \given y)$ is a probability measure,
    $\nu_i(\Pp\setminus\rho_i^{-1}(y)\given y) = 0$, and $\nu(\dd\pp) = \int_{(0,1)} \nu_i(\dd\pp \given y) \nu_i(\dd y)$.
    We define first
    \[ h_i(y) = \begin{cases}
                    (1-y)^2 , & i = 1, \\
                    y^2, & i \ge 2.
                  \end{cases}
    \]
    Thus, $h_i(p_i) \le (1-p_1)^2$ for all $i$ and $\pp$, and in particular
    $\int_{\Pp} h_i(p_i)\, \nu(\dd\pp) = \int_{(0,1)} h_i(y) \, \nu_i(\dd y) < \infty$.
    Define $\bar{\lambda}_i(\dd\pp) = h_i(p_i)\nu(\dd\pp)/\int h_i\, \dd \nu_i$, a probability measure.
    Then by standard results on disintegration of measures (see \cite{Sim}, for instance)
    there exist probability kernels $\nu_i(\dd\pp \given y)$
    such that $\nu_i(\Pp\setminus\rho_i^{-1}(y)\given y) = 0$ and
    \begin{eqnarr*} \int_{\Pp} \frac{f(\pp)h_i(p_i)}{\int h_i\,\dd\nu_i} \, \nu(\dd\pp)
      = \int_{\Pp} f(\pp) \, \bar{\lambda}_i(\dd\pp)
      &=& \int_{(0,1)} \int_{\Pp} f(\pp) \, \nu_i(\dd\pp \given y) \, \bar{\lambda}_i\circ\rho_i^{-1}(\dd y)
      \\
      &=& \int_{(0,1)} \int_{\Pp} \frac{f(\pp) h_i(p_i)}{\int h_i\, \dd\nu_i} \, \nu_i(\dd\pp \given y) \, \nu_i(\dd y).
    \end{eqnarr*}
    This completes the proof.
    }
    \end{proof}
  \end{lem}
  
  Finally we recall the notion of randomisation.
  If $M = \sum_{k} \delta_{s_k}$
  is a random point measure on a space $S$, and $q$ is a kernel from
  $S$ into a space $T$, then the $q$-randomisation of $M$ is
  the measure $\sum_{k} \delta_{s_k,t_k}$, where, 
  conditional on $M$ and independently for each $k$,
  $t_k$ is a random element of $T$
  chosen according to the measure $q(s_k,\cdot)$.
  We refer to \cite[Ch.~12, p.~226]{Kal-found} for a complete
  definition and properties.
  
  We are now in a position to define all the relevant quantities, and assemble
  them into the decorated process.
  
  \begin{defn}
  	\fakephantomsection\label{d:eta} 
    \begin{enumerate}	    
      \item
        Let $\mathtt{M}(\dd s, \dd z)$ be the
        jump measure of $\xi$, that is, 
        a Poisson random measure with intensity $\dd s \,\Pi(\dd z)$.
        Define $M$ to be the pushforward of $\mathtt{M}$ under
        the function $(s,z) \mapsto (s,e^z)$. Thus, $M(\dd s,\dd y)$ is a Poisson
        random measure with intensity $\dd s\, \pi(\dd y)$.
      \item 
        Now let $\lambda_i(\dd y) = y^\omega \nu(p_i\in \dd y)$.
        Observe that $\lambda_i$ is absolutely continuous
        with respect to $\pi$, and define $g_i = \dd \lambda_i / \dd \pi$, the Radon-Nikodym derivative.
        In particular, $\sum_{i\ge 1} g_i \equiv 1$ $\pi$-a.e..
      \item
        We define a probability kernel from $[0,\infty)\times (0,1)$ to $\NN\times \Pp$ by
        \[ q(s,y,\dd i,\dd \pp) = g_i(y) \nu_i(\dd\pp \mid y) \zeta(\dd i) , \]
        where $\zeta$ is counting measure on $\NN$,
        and a measure $\eta$ on $[0,\infty)\times (0,1)\times\NN\times \Pp$ by
        \[ \eta(\dd s, \dd y, \dd i, \dd \pp) \coloneqq q(s,y,\dd i,\dd\pp) \pi(\dd y) \dd s
          = y^\omega \nu(p_i\in \dd y) \nu_i(\dd \pp\mid y) \dd s \zeta(\dd i) . \]
        This has the following consistency properties:
        \[ 
          \int_{\NN\times \Pp} \eta(\dd s, \dd y, \dd i, \dd \pp)
          = \pi(\dd y) \dd s \quad \text{and} \quad 
          \int_{(0,1)\times \NN} \eta(\dd s, \dd y, \dd i, \dd \pp)
          = 
          \sum_{i\ge 1} p_i^{\omega}\nu(\dd \pp)\dd s .
        \]
      \item
        Let $N(\dd s, \dd y, \dd i, \dd \pp)$ be the $q$-randomisation of $M$.
        It is readily checked (via \cite[Lemma 12.2]{Kal-found}, say)
        that $N$ is a Poisson random measure with
        intensity $\eta(\dd s, \dd y, \dd i, \dd \pp)$.
    \end{enumerate}
  \end{defn}

  This completes the definition of the decorations, and we will now
  define a process $\Yy = (\Yy(t), t\ge 0)$.
  We regard the definition as being given under the probability measure $\lQt$,
  and we assume that the underlying probability space has been enlarged as required to
  accommodate it.
  \begin{defn}\label{d:decorated-spine}
    Let $(\Zz^{[s,j]})_{s\in \RR, j\in \NN}$ denote a collection of independent
    branching L\'evy processes with triple $(a,\sigma,\nu)$. Under the probability measure
    $\lQt$, the \emph{decorated Lévy process} $\Yy$, with parameters $(a,\sigma,\nu,\omega)$,
    is defined
    as follows:
    \[ \Yy(t) = \delta_{\xi(t)} + \int_{[0,t]\times(0,1)\times\NN\times\Pp}
      \sum_{j \ne i} \mbigl[\Zz^{[s,j]}(t-s) + \xi(s-) + \log p_j \mr] \, N(\dd s,\dd y, \dd i, \dd \pp) , \qquad t\ge 0,
    \]
    where the sum appearing on the right-hand side is over only those $j$ for which
    $p_j > 0$. The summand has the following interpretation:
    if $\mu = \sum_{i \in I} \delta_{\mu_i}$ is a point measure and $z \in \RR$, then
    $\mu + z \coloneqq \sum_{i\in I} \delta_{\mu_i +z} $.
    See \autoref{f:Y}.
  \end{defn}

  \begin{figure}[t]
      \centering
      \includegraphics[width=\textwidth]{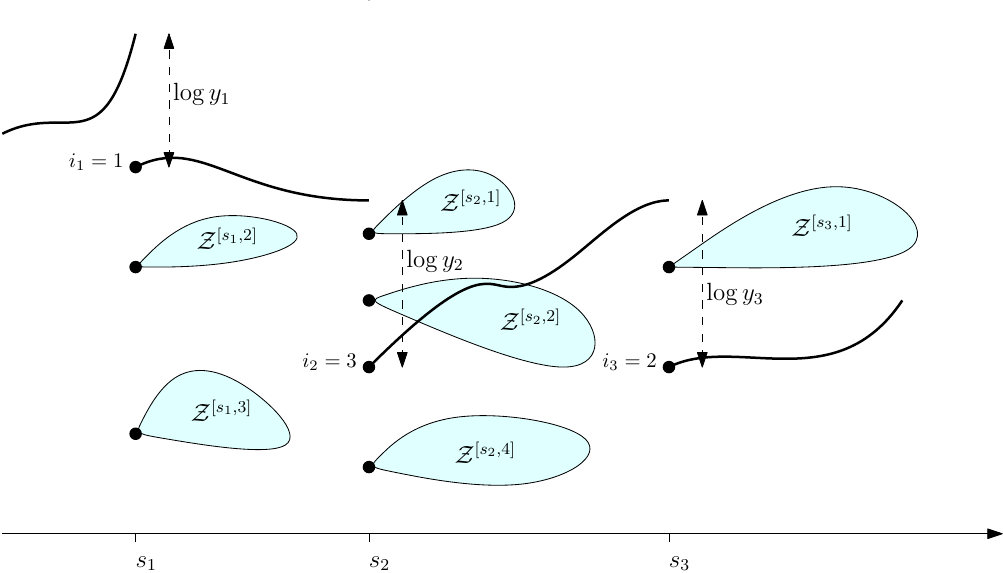}
      \caption{A sketch of the construction of the decorated L\'evy process.
      One begins with the thick black line, the path of the L\'evy
      process $\xi$,
      whose jumps determine the points $(s_k,y_k)$ which are the support
      of $M$. Given these
      points, one adds decorations which are the
      $\pp^{(k)}$ (represented by the solid disks)
      and the determination of the $i_k$.
      At given time and space points, determined by these decorations, one
      introduces independent copies $\mathcal{Z}^{[s_k,j]}$
      of the branching L\'evy process
      (the shaded shapes.)
      } 
      \label{f:Y}
  \end{figure}

  \medskip\noindent
  Let us consider the process $\Yy$ under truncation. Formally, this is required
  to give the particles labels; however, the truncated processes
  will also be a vital component in showing the equivalence of the two spine constructions.
  
  Let $b>0$, and recall that $k_b$ is given by \eqref{e:kb}.
  We define a random measure $N_b$
  by the mapping
  \begin{equation}
    \int N_b(\dd s,\dd y, \dd i, \dd\pp) f(s,y,i,\pp) = \int N(\dd s,\dd y, \dd i, \dd \pp) f(s,y,i, k_b(\pp)).
    \label{e:Nb}
  \end{equation}
  Let
  \[ A_b = \{ (y,i,\pp) : i \ge 2 \text{ and } y \le e^{-b} \} \]
  and define the first entry time
  \[ \tau_b \coloneqq \inf\{ t\ge 0: N(\{t\}\times A_b) > 0 \} ,\]
  which is a stopping time in the natural filtration of $N$.
  Then $\tau_{b}$ is the time at which the spine is killed under truncation at level $b$, and it has an exponential distribution with parameter 
  \[ \theta_b \coloneqq \int_{\Pp} \sum_{i\ge 2}p_i^\omega \Indic{p_i\le e^{-b}} \, \nu(\dd\pp) < \infty . \]
  We define the process $\Yy^{(b)}$ by the expression
  \begin{multline} \Yy^{(b)}(t) = \delta_{\xi(t)} \Indic{t<\tau_b} \\
    {} + \int_{[0,t]\times (0,1)\times \NN\times\Pp} N_b(\dd s,\dd y, \dd i,\dd\pp) 
    \sum_{j\ne i} \mbigl[ (\Zz^{[s,j]})^{(b)}(t-s) + \xi(s-)+\log p_j\mr] \Indic{s\le \tau_b},
    \label{e:Yb}
  \end{multline}
  where $(\Zz^{[s,j]})^{(b)}$ indicates that the immigrated copy of $\Zz$ is truncated at level $b$.
  
  With this definition, we have
  all the processes $\Yy^{(b_n)}$, for $n\ge 0$, defined on the same probability space as $\Yy$.
  Moreover, following \autorefref{r:misc}{i:misc:3}, we also have the processes
  $\bar\Yy^{(b_n)}$ all defined on the same space.
  Now suppose that $m<n$, and denote by $(\Yy^{(b_n)})^{(b_m)}$ the result of applying
  the truncation method of \eqref{e:bnbm} to the process $\Yy^{(b_n)}$.
  It follows that $(\Yy^{(b_n)})^{(b_m)} = \Yy^{(b_m)}$ almost surely;
  this can be verified by comparing the particles present at the first braching time $T_{b_n} \coloneqq \inf\{t\ge 0: \card\Yy^{(b_n)} \ge 2 \}$,
  and then proceeding iteratively.
  Thus, we have that, for every $t\ge 0$, $\Yy(t) = \lim_{n\to\infty} \Yy^{(b_n)}(t)$ as an increasing
  almost sure limit,
  with $\bar\Yy(t)$ being defined similarly.
  
  \skippar
  We now specify a \emph{distinguished line of descent} in $\bar\Yy$, which we denote by 
  $V = (V_t, t\ge 0)$ with $V_t \in \tree$. We want it to track the particle whose
  position is given by $\xi$, and it may be found explicitly as follows.
  
  Fix $t\ge 0$. Observe that
  \[ \int_{[0,t]\times(0,1)\times\NN\times\Pp} \Indic{i \ne 1} \eta(\dd s,\dd y,\dd i,\dd \pp) = t\int_{\Pp} \sum_{i\ge 2} p_i^\omega \,\nu(\dd \pp) < \infty , \]
  and so we may enumerate the atoms of $N\rvert_{[0,t]\times(0,1)\times(\NN\setminus\{1\})\times\Pp}$ 
  as $(s_j,y_j,i_j,\pp^{(j)})_{1\le j\le J(t)}$, with $0\le J(t)<\infty$ and $s_j< s_{j+1}$;
  moreover let $s_{J(t)+1}=t$.
  If $J(t)=0$, that is, $N$ restricted to ${[0,t]\times(0,1)\times(\NN\setminus\{1\})\times\Pp}$ has no atoms,
  we set $V_t = \varnothing$.
  Otherwise, we proceed by recursion as follows.
  Let $V_0 = \varnothing$. 
  For $j\ge 1$, we consider the children of particle $V_{s_{j-1}}$ which are born at time $s_j$.
  Among these offspring, we pick $V_{s_j}$ such that $\Yy_{V_{s_j}}(s_j) = \xi(s_j-)+\log y_j$.
  To be entirely explicit, recall the definition of $\Lu$ from \eqref{e:L},
  define
  $k_j = i_j - \card\{ k \ge 1: \Lu(p_k) < \Lu(p_{i_j})\}$,
  and let  
  $V_{s_j} = V_{s_{j-1}}(\Lu(y_j),K_{V_{s_{j-1}}}(s_j,\Lu(y_j)),k_j)$.
  For $s \in [s_{j-1},s_j)$ and $j\ge 1$, we let $V_s = V_{s_{j-1}}$.
  Thus, in particular, $V_t = V_{s_{J(t)}}$.
  
  \begin{defn}\label{d:deco2}
    The \emph{decorated spine process} with parameters $(a,\sigma,\nu,\omega)$ 
    is the process $(\bar\Yy(t),V_t)_{t\ge 0}$ under the measure $\lQt$.
  \end{defn}
  
  We remark that, by its construction,
  $(\bar{\Yy},V) = ((\bar\Yy(t),V_t), {t\ge 0})$ is a Markov process,
  and in particular it possesses a branching property exactly analogous to \autoref{l:Zbar-U-bp}.
  Moreover, it has similar regularity properties, as we now show.
  We need the following
  lemma, whose proof is quite technical but requires nothing more than
  the definition of $\Yy$ and an understanding of the additive martingale
  $W(\omega,\cdot)$
  of $\Zz$.
  \begin{lem}\label{l:cv-Ybn}
    For $q\in(\dom \kappa)^{\circ}\cap (1, \infty)$ and $t\ge 0$, 
    $\sup_{s\le t} d_q(\bar{\Yy}(s),\bar{\Yy}^{(b_n)}(s)) \to 0$ in probability
    as $n\to\infty$.
    \begin{proof}
      In the proof, we will use similar notation ($\tree_s$, $K_u(s)$, etc.)
      for the atoms of $\bar{\Yy}$ to that which we used for the atoms of $\bar{\Zz}$.
    
      The proof follows very similar lines to the proof of \autoref{l:cv-Zbn},
      and we again begin by using the triangle inequality to obtain
      \begin{multline}\label{e:cadlag-Y-inter}
        d_q(\bar\Yy(s),\bar\Yy^{(b_n)}(s))
        \le
        d_q\mBigl( \sum_{u \in \tree_s} \delta_{(u,K_u(s),\Yy_u(s))},
        \sum_{u\in\tree_s^{(b_n)}} \delta_{(u,K_u(s),\Yy_u(s))} \mr)
        \\
        {} +
        d_q\mBigl( \sum_{u \in \tree_s^{(b_n)}} \delta_{(u,K_u(s),\Yy_u(s))},
        \sum_{u\in\tree_s^{(b_n)}} \delta_{(u,K_u^{(b_n)}(s),\Yy_u(s))} \mr).
      \end{multline}
      To show that the second term vanishes as $n\to\infty$, we can use the same method as in \autoref{l:cv-Zbn}, which gives in the first instance,
      for $r\ge 0$,
      \begin{multline}
        \lEQ\biggl[
        \sup_{s\le t} d^{(r)}\mBigl( \sum_{u \in \tree_s^{(b_n)}} \delta_{(u,K_u(s),\Yy_u(s))},
        \sum_{u\in\tree_s^{(b_n)}} \delta_{(u,K_u^{(b_n)}(s),\Yy_u(s))} \mr)
        \biggr] \\
        \le
        \sum_{l\ge n+1} \exp\bigl(-(l+e^{4b_l})\bigr) \lEQ[\#\Yy^{(b_l)}(t)].
        \label{e:cadlag-Y-inter2}
      \end{multline}
      Therefore, we need to bound $\lEQ[\# \Yy^{(b_n)}(t)]$. We do this as follows,
      beginning with:  
      \begin{eqnarr*}
        \lEQ[ \# \Yy^{(b_n)}(t)]
        &=&
        \lEQ \biggl[ \Indic{t<\tau_{b_n}}
        + \int_{[0,t]\times (0,1)\times \NN \times \Pp}
        N_{b_n}(\dd s, \dd y, \dd i, \dd \pp)
        \sum_{j\ne i} \# \bigl(\Zz^{[s,j]}\bigr)^{(b_n)}(t-s) \Indic{s\le \tau_{b_n}}
        \biggr] \\
        &=& \lQ(t<\tau_{b_n})
        + \lEQ\biggl[ \int_{[0,\tau_{b_n}]\times A_{b_n}^c}
        \eta_{b_n}(\dd s, \dd y, \dd i, \dd \pp)
        \sum_{j \ne i} \lE[\# \Zz^{(b_n)}(t-s) ]\biggr],
      \end{eqnarr*}
      where in the final equality we use the fact that $N_{b_n}$ restricted to 
      $[0,\infty)\times A_{b_n}^c$ is independent of $\tau_{b_n}$, together with the
      compensation formula for the Poisson random measure $N_{b_n}$ with 
      intensity measure $\eta_{b_n}$,
      which is $\eta$ restricted to $[0,\infty)\times A_{b_n}^c$.
      
      Recall that $\tau_b$ is an exponentially distributed
      random variable with rate $\theta_b$.
      Moreover,
      $\eta_b\rvert_{[0,\infty)\times A_b^c} = \eta^{(b)}$, where
      $\eta^{(b)}$ is  
      the measure $\eta$ constructed as in \autoref{d:eta} for the parameters
      $(a,\sigma,\nu^{(b)},\omega)$, that is,
      \begin{equation}\relax\label{e:eta-n}
        \eta^{(b)}(\dd s, \dd y, \dd i, \dd \pp)
        = y^\omega \nu^{(b)}(p_i\in \dd y) \nu^{(b)}_i(\dd \pp\mid y) \dd s \zeta(\dd i).
      \end{equation}
      Thus we can rewrite the previous expression to obtain that
      \begin{eqnarr*}
        \lEQ[\# \Yy^{(b_n)}(t)]
        &=&
        e^{-t \theta_{b_n}}
        + \int_{[0,t] \times (0,1)\times \NN\times \Pp}
        \eta^{(b_n)}(\dd s, \dd y, \dd i, \dd \pp)
        (\#\pp - 1) \lE[\# \Zz^{(b_n)}(t-s)] e^{-s \theta_{b_n}},
      \end{eqnarr*}
      where $\# \pp$ is the number of non-zero elements in $\pp$.
      Continuing to evaluate the components of this expression, we
      obtain
      \begin{eqnarr*}
        \lEQ[\# \Yy^{(b_n)}(t)]
        &=& e^{-t\theta_{b_n}}
        + \int_0^t e^{(t-s)\kappa^{(b_n)}(0)} e^{-s\theta_{b_n}}
        \cdot \int_{\Pp} (\# \pp - 1) \sum_{i\ge 1} p_i^\omega \, \nu^{(b_n)}(\dd \pp).
      \end{eqnarr*}
      We observe that
      \[
      \int_{\Pp} (\# \pp - 1) \sum_{i\ge 1} p_i^\omega \, \nu^{(b_n)}(\dd \pp)
      = \int_{\Pp \setminus \Pp_1}\Indic{p_1 > e^{-b_n}} \mbigl( \sum_{i\ge 2} \Indic{p_i > e^{-b_n}}\mr) \mbigl(p_1^{\omega}+\sum_{i\ge 2} p_i^\omega \mr)\, \nu^{(b_n)}(\dd \pp). 
      \]
      If $\omega \ge 0$, then $p_1^\omega \le 1$,
      whereas if $\omega<0$, then
            $\Indic{p_1>e^{-b_n}}p_1^\omega \le e^{-\omega b_n}$.
      In either case, we have
      \begin{eqnarr*}
      \int_{\Pp \setminus \Pp_1}\Indic{p_1 > e^{-b_n}} \mbigl( \sum_{i\ge 2} \Indic{p_i > e^{-b_n}}\mr) p_1^{\omega}\, \nu^{(b_n)}(\dd \pp) 
     & \le& \max(1, e^{-\omega b_n} ) \int_{\Pp} \mbigl( \sum_{i\ge 2} \Indic{p_i > e^{-b_n}}\mr) \,\nu^{(b_n)}(\dd\pp) \\
      &=&\max(1, e^{-\omega b_n} )  \kappa^{(b_n)}(0).
      \end{eqnarr*}
     The assumption $\omega \in \dom\kappa$ implies that
     $\int_{\Pp} \sum_{j\ge 2} p_j^\omega \nu^{(b_n)}(\dd\pp) \le \int_{\Pp} \sum_{j\ge 2} p_j^\omega \nu(\dd\pp) <\infty$,
	   and the fact that $\sum_{i\ge 1} p_i\le 1$ yields
     $\sum_{i\ge 2} \Indic{p_i> e^{-b_n}} \le e^{b_n}$; thus, we obtain 
	\[\int_{\Pp \setminus \Pp_1}\Indic{p_1 > e^{-b_n}} \mbigl( \sum_{i\ge 2} \Indic{p_i > e^{-b_n}}\mr) \mbigl(\sum_{i\ge 2} p_i^\omega \mr)\, \nu^{(b_n)}(\dd \pp)
	\le e^{b_n} \int_{\Pp} \sum_{i\ge 2} p_i^{\omega} \nu(\dd\pp).\]
	It follows that
      \[
        \lEQ[\# \Yy^{(b_n)}(t)]
        \le
        e^{-t\theta_{b_n}} 
        + \frac{1-e^{-t(\kappa^{(b_n)}(0) + \theta_{b_n})}}{\kappa^{(b_n)}(0) + \theta_{b_n}}
        e^{t\kappa^{(b_n)}(0)} 
        \cdot
        \mBigl(
        \max(1, e^{-\omega b_n} )\kappa^{(b_n)}(0)
        + e^{b_n} \int_{\Pp} \sum_{i\ge 2} p_i^{\omega} \nu(\dd\pp)
        \mr).
      \]
      Recall from the proof of \autoref{l:cv-Zbn} that $\kappa^{(b_n)}(0)<C e^{2 b_n}$
      for some $C>0$ depending only on $\nu$; thus, for some $C'>0$, we have
      \[
      \lEQ[\# \Yy^{(b_n)}(t)]\le  e^{t C e^{2 b_n} + C' b_n}.
      \] 
      This implies that the right-hand side of \eqref{e:cadlag-Y-inter2},
      and thence the second term of \eqref{e:cadlag-Y-inter}, converges
      to $0$ as $n\to\infty$.
      
      We turn now to the first term of \eqref{e:cadlag-Y-inter}.
      Using the same trick as in \eqref{e:trick}, we select
      $q'$ arbitrarily such that $q>q'$ and $q' \in (\dom \kappa)^\circ$,
      and obtain
      \begin{eqnarr*}
        d_q\mBigl( \sum_{u \in \tree_s} \delta_{(u,K_u(s),\Yy_u(s))},
        \sum_{u\in\tree_s^{(b_n)}} \delta_{(u,K_u(s),\Yy_u(s))} \mr)
        &\le&
        \frac{q}{q-q'}
        \sum_{u \in \tree_s\setminus \tree_s^{(b_n)}} e^{q' \Yy_u(s)}.
      \end{eqnarr*}
      We now use the definition of $\Yy$ and $\Yy^{(b_n)}$ to write
      $
        \sum_{u \in \tree_s\setminus \tree_s^{(b_n)}} e^{q' \Yy_u(s)}
        = I_1(s) + I_2(s) + I_3(s),
      $
      where for reasons of brevity the terms $I_i$ will be defined as we proceed.
      The first of these is
      \begin{eqnarr*}
        I_1(s)
        &=&
        \int_{[0,s]\times (0,1)\times \NN\times \Pp} N(\dd v, \dd y,\dd i, \dd \pp)
        \sum_{j\ne i} \sum_{u \in \tree_{s-v}^{[v,j]}}
        e^{q' \bigl[ \Zz_u^{[v,j]}(s-v) + \log p_j + \xi(v-) \bigr]}
        \\
        &&
        {} -
        \int_{[0,s]\times (0,1)\times \NN\times \Pp} N(\dd v, \dd y,\dd i, \dd \pp)
        \sum_{j\ne i} \sum_{u \in \tree_{s-v}^{[v,j]}}
        e^{q' \bigl[ (\Zz_u^{[v,j]})^{(b_n)}(s-v) + \log p_j + \xi(v-) \bigr]}.
      \end{eqnarr*}
      If we define
      $M_{v,j}^{(n)}(t)
      = \sum_{u\in \tree_t} e^{q' \Zz^{[v,j]}_u(t)} - \sum_{u\in \tree_t^{(b_n)}} e^{q' \Zz^{[v,j]}_u(t)}$
      for arbitrary $v\ge 0$ and $j \ge 1$,
      and observe that this is a non-negative martingale in its own filtration,
      it then follows that
      \begin{eqnarr*}
        \eqnarrLHS{\lEQ\mbigl[ \min\mbigl\{ \sup_{s\le t}
        I_1(s) , 1\mr\} \mr]}
        &\le&
        \lEQ\mbiggl[
        \int_{[0,t]\times (0,1)\times \NN\times \Pp} N(\dd v, \dd y,\dd i, \dd \pp)
        e^{q'\xi(v-)}
        \sum_{j\ne i}
        p_j^{q'}
        \min\bigl\{
        \sup_{s\le t}
        M_{v,j}^{(n)}(s-v)
        ,1\bigr\}
        \mr]
        \\
        &\le&
        \lEQ\mbiggl[
        \int_{[0,t]\times (0,1)\times \NN\times \Pp} N(\dd v, \dd y,\dd i, \dd \pp)
        e^{q'\xi(v-)}
        \sum_{j\ne i}
        p_j^{q'}
        \min
        \bigl\{\sup_{w\le t} M_{v,j}^{(n)}(w) , 1 \bigr\}
        \mr]
        \\
        &=&
        \int_0^t e^{v \Ess_\omega \kappa(q')} \, \dd v
        \cdot
        \int_{\Pp} \sum_{i\ge 1} \sum_{j\ne i} p_i^\omega   p_j^{q'} \, \nu(\dd \pp)
        \cdot
        \lE\biggl[ 
        \min \bigl\{ \sup_{w\le t} M_{0,1}^{(n)}(w) ,1\bigr\}
        \biggr].\IEEEyesnumber \label{e:cadlag-Y-I1}
      \end{eqnarr*}
      We first claim that if $q' \ge 1$ and $q'\in \dom \kappa$, then
      \begin{equation}\label{e:bd-int-N}
        \int_{\Pp} \sum_{i\ge 1} \sum_{j\ne i} p_i^\omega   p_j^{q'} \, \nu(\dd \pp)
        =
        \int_{\Pp\setminus \Pp_1} \sum_{i\ge 1} \sum_{j\ne i} p_i^\omega   p_j^{q'} \, \nu(\dd \pp)
        < \infty.
      \end{equation}
      We begin with the estimate
      \[
        \sum_{i\ge 1} \sum_{j\ne i} p_i^\omega p_j^{q'}
        \le
        2p_1^\omega \sum_{j\ge 2} p_j^{q'}
        + \Bigl( \sum_{i\ge 2} p_i^\omega\Bigr)\Bigl( \sum_{i\ge 2} p_i^{q'}\Bigr) .
      \]
      If $\omega \ge 0$, then
      \eqref{e:bd-int-N} follows from
      the fact that $q'\in \dom \kappa$ and $q' \ge 1$.
      If $\omega < 0$, then since $\pp\in\Pp\setminus\Pp_1$,
      we have $p_1^\omega \le p_2^\omega \le \sum_{i\ge 2} p_i^\omega$
      and \eqref{e:bd-int-N} again follows.
       
       
       Finally, using       
       Doob's maximal inequality just as in \autoref{l:cv-Zbn},
       we see that $\sup_{w\le t} M_{0,1}^{(n)}(w)$ converges to $0$ in probability as $n\to\infty$.
       Thus, the right-hand side of \eqref{e:cadlag-Y-I1} approaches $0$ also,
       and so $\sup_{s\le t} I_1(s)$ tends to $0$ in probability.

      This deals with the term $I_1$, which is the main difficulty. The term
      $I_2$ is defined as
      \begin{eqnarr*}
        I_2(s)
        &=&
        \int_{[0,s]\times (0,1)\times \NN\times \Pp} N(\dd v, \dd y,\dd i, \dd \pp)
        \sum_{j\ne i} \sum_{u \in \tree_{s-v}^{[v,j]}}
        e^{q' \bigl[ (\Zz_u^{[v,j]})^{(b_n)}(s-v) + \log p_j + \xi(v-) \bigr]}
        \\
        &&
        {} -
        \int_{[0,s]\times (0,1)\times \NN\times \Pp} N_{b_n}(\dd v, \dd y,\dd i, \dd \pp)
        \sum_{j\ne i}
        \sum_{u \in \tree_{s-v}^{[v,j]}}
        e^{q' \bigl[ (\Zz_u^{[v,j]})^{(b_n)}(s-v) + \log p_j + \xi(v-) \bigr]}
        \\
        &=&
        \int_{[0,s]\times (0,1)\times \NN\times \Pp} N(\dd v, \dd y,\dd i, \dd \pp)
        \sum_{j\ne i} p_j^{q'} \Indic{j\ne 1 \text{ and } p_j<e^{-b_n}} \sum_{u \in \tree_{s-v}^{[v,j]}}
        e^{q' \bigl[ (\Zz_u^{[v,j]})^{(b_n)}(s-v) + \xi(v-) \bigr]}.
      \end{eqnarr*}
      Using a similar technique to the one for the term $I_1$, we obtain
      \begin{multline}
        \lEQ\bigl[ \min\bigl\{ \sup_{s\le t} I_2(s) ,1 \bigr\}\bigr]
        \\ {} \le
        \int_0^t e^{v\Ess_\omega \kappa(q')} \, \dd v
        \cdot
        \int_{\Pp} \sum_{i\ge 1} \sum_{j\ne i} p_i^\omega p_j^{q'}
        \Indic{j\ne 1 \text{ and } p_j<e^{-b_n}}
        \, \nu(\dd \pp)
        \cdot
        \lE\mbigl[ \min\bigl\{\sup_{r\le t} W(q',r),1\bigr\} \mr].
        \label{e:cadlag-Y-I2}
      \end{multline}
      We can then make the estimate 
      \begin{equation}\label{e:bd-supW}
        \lE\mbigl[ \min\bigl\{\sup_{r\le t} W(q',r),1\bigr\} \mr]
        \le \lE\mbigl[ \min\bigl\{\sup_{r\le t} W(q',r),1\bigr\}^2 \mr]^{1/2}
        \le \lE[W(q',t)]^{1/2} < \infty,
      \end{equation}
      where in the second inequality,
      we use a variation on Doob's $L^2$-%
      inequality (see the proof of Corollary II.1.6 in \cite{RY-cmbm}.)
      Moreover, take $\epsilon>0$ such that $q'-\epsilon>1$ and $q'-\epsilon \in \dom \kappa$, then
      \[
        \sum_{i\ge 1} \sum_{j\ne i} p_i^\omega p_j^{q'}
        \Indic{j\ne 1 \text{ and } p_j<e^{-b_n}}
        \le
        e^{-\epsilon b_n} \sum_{i\ge1}\sum_{j\ne i} p_i^\omega p_j^{q'-\epsilon },
      \]
      and just as in the $I_1$ case, we know that
      $\int_{\Pp} \sum_{i\ge 1}\sum_{j\ne i} p_i^\omega p_j^{q'-\epsilon } \,\nu(\dd\pp) <\infty$.
      It follows that the right hand side of \eqref{e:cadlag-Y-I2} tends to zero,
      and thus $\sup_{s\le t} I_2(s) \to 0$ in probability.
      
      Lastly, we turn to $I_3$. This term is defined as
      \begin{eqnarr*}
        I_3(s)
        &=&
        \int_{[0,s]\times (0,1)\times \NN\times \Pp} N_{b_n}(\dd v, \dd y,\dd i, \dd \pp)
        \sum_{j\ne i}
        \sum_{u \in \tree_{s-v}^{[v,j]}}
        e^{q' \bigl[ (\Zz_u^{[v,j]})^{(b_n)}(s-v) + \log p_j + \xi(v-) \bigr]}
        \\
        &&
        {} -
        \int_{[0,s]\times (0,1)\times \NN\times \Pp} N_{b_n}(\dd v, \dd y,\dd i, \dd \pp)
        \sum_{j\ne i}
        \sum_{u \in \tree_{s-v}^{[v,j]}}
        e^{q' \bigl[ (\Zz_u^{[v,j]})^{(b_n)}(s-v) + \log p_j + \xi(v-) \bigr]}
        \Indic{v\le \tau_{b_n}}
        \\
        &=&
        \int_{(\tau_{b_n},s]\times(0,1) \times \NN\times \Pp}
        N_{b_n}(\dd v, \dd y,\dd i, \dd \pp)
        e^{q'\xi(v-)} \sum_{j\ne i} p_j^{q'}
        \sum_{u \in \tree_{s-v}^{[v,j]}} e^{q'(\Zz_u^{[v,j]})^{(b_n)}(s-v)}.        
      \end{eqnarr*}
      In particular,
      \begin{eqnarr*}
        \sup_{s\le t} I_3(s)
        &\le&
        \int_{(\tau_{b_n},t]\times(0,1) \times \NN\times \Pp}
        N_{b_n}(\dd v, \dd y,\dd i, \dd \pp)
        e^{q'\xi(v-)} \sum_{j\ne i} p_j^{q'}
        \sup_{w\le t}
        \sum_{u \in \tree_{w}^{[v,j]}} e^{q'(\Zz_u^{[v,j]})^{(b_n)}(w)} .  
      \end{eqnarr*}
      Making a change of variable in the integral, and using the independence
      properties of the Poisson point process $N_{b_n}$, we obtain
      \begin{IEEEeqnarray*}{rCll}
        \IEEEeqnarraymulticol{4}{l}{
          \lEQ\mbigl[\min\bigl\{\sup_{s\le t} I_3(s),1\bigr\}\mr]
        } \\
        &\le&
        \lEQ\biggl[
        &
        \Indic{\tau_{b_n} \le t}
        e^{q' \xi(\tau_{b_n})}
        \\
        &&& {} \cdot
        \lEQ\biggl[
        \int_{(r,t]\times (0,1) \times \NN\times \Pp}
        N_{b_n}(\dd v, \dd y,\dd i, \dd \pp)
        e^{q' \xi(v-)}
        \sum_{j\ne i}
        p_j^{q'}
        \min\bigl\{ \sup_{w \le t} e^{q'(\Zz_u^{[v,j]})^{(b_n)}(w)} ,1 \bigr\}
        \biggr]_{r = \tau_{b_n}}
        \biggr] \\
        &\le&
        \IEEEeqnarraymulticol{2}{l}{
         \lEQ[ e^{q' \xi(\tau_{b_n})} \Indic{\tau_{b_n} \le t}]
         \cdot
        \int_0^t e^{v\Ess_\omega \kappa(q')} \, \dd v
        \cdot
        \int_{\Pp} \sum_{i\ge 1} \sum_{j\ne i} p_i^\omega p_j^{q'}
        \, \nu(\dd \pp)
        \cdot
        \lE\mbigl[ \min\bigl\{\sup_{r\le t} W(q',r),1\bigr\} \mr]
      }.
      \end{IEEEeqnarray*}
      By \eqref{e:bd-int-N} and \eqref{e:bd-supW}, we are left with just the first expectation, for which we have:
      \begin{eqnarr*}
        \lEQ[\Indic{\tau_{b_n} \le t} e^{q'\xi(\tau_{b_n})} ]
        &\le&
        \lEQ[\Indic{\tau_{b_n} \le t} \sup_{s\le t} e^{q'\xi(s)} ]
        \le
        \lQ(\tau_{b_n} \le t)^{1/2}
        \lE\bigl[ \bigl(\sup_{s\le t} e^{q'\xi(s)} \bigr)^2\bigr]^{1/2}.
      \end{eqnarr*}
      The second term on the right-hand side may be bounded using Doob's $L^2$-%
      inequality for the exponential martingale of the L\'evy process $\xi$;
      and the first term approaches zero as $n\to\infty$ since
      $\tau_{b_n}$ has an exponential distribution whose parameter approaches zero.
      It follows that $\sup_{s\le t} I_3(s) \to 0$ in probability.
      
      Having shown the necessary convergence for each term $I_i$, we have now proved
      that the first term in \eqref{e:cadlag-Y-inter} converges to zero in probability,
      and this
      completes the proof.
    \end{proof}
  \end{lem}
  
  An immediate consequence of the lemma is the regularity of $\bar{\Yy}$.
  \begin{cor}
  Let $q\in (\dom \kappa)^{\circ}\cap (1,\infty)$. Then the process $\bar{\Yy}$ possesses a c\`adl\`ag version in $\mbigl(\Mmp(\mathbf{X}), d_q\mr)$.
  \end{cor}
  
  We will fix from now on a metric $d_q$ with $q\in (\dom \kappa)^{\circ}\cap (1, \infty)$, and assume that the process $\bar{\Yy}$ is c\`adl\`ag.
  
  \section{The spine decomposition theorem}
  \label{s:fb}
  
  We now show that the forward and backward constructions of the process
  with distinguished spine under $\lQ$, i.e.\ $(\bar \Zz(t),U_t)_{t\ge 0}$
  and $(\bar \Yy(t),V_t)_{t\ge 0}$,
  in fact have the same law.
 
  We use a truncation technique,
  recalling the definitions of $\Zz^{(b)}$, $\nu^{(b)}$ 
  and the sequence $(b_n)$ from \autoref{s:construct-truncate}.
  In order to simplify notation in the proof, we define the measure $\lPn$ such that
  the law of $(\bar \Zz,\tree_\cdot)$ under $\lPn$ is that of  $(\bar \Zz^{(b_n)},\tree^{(b_n)}_\cdot)$ under $\lP$.
  For $n\ge 1$, we consider on the one hand the
  measure $\lQn$ constructed from $\lPn$ as follows: 
  \begin{equation}\label{e:Qn}
    \lEQn[ F(\bar \Zz(s), s \le t) \Indic{U_t=u}] 
    = e^{-t\kappa^{(b_n)}(\omega)} \lE^{(b_n)}[ F(\bar \Zz(s), s \le t) e^{\omega \Zz_u(t)}],
  \end{equation}
  where $F$ is a continuous bounded functional on $D([0,\infty),\mathcal{M}_p(\mathbf{X}))$, 
  and we use the convention $e^{\Zz_u(t)} = 0$ if $u\notin \tree_t$.
  On the other hand, we
  regard the process $(\bar \Yy,V)$ under $\lQtn$ as being the decorated spine process
  with parameters $(a,\sigma,\nu^{(b_n)},\omega)$.
  
  \begin{lem}\relax\label{l:fb-n}
    Under $\lQn$, $(\bar \Zz(t),U_t)_{t\ge 0}$ is equal in law to $(\bar \Yy(t),V_t)_{t\ge 0}$.
  \end{lem}
  \begin{proof}
    We verify that the two processes have the same decomposition at the first branching time;
    since both satisfy a branching property, this is sufficient.

    We start with $(\bar \Zz,U)$ under $\lQn$. Let $T$ denote the time of the first branching event, that is,
    \[ T = \inf\{ t\ge 0 : \card \bar \Zz(t) \ge 2 \}, \]
    where $\card\bar\Zz(t) = \bar\Zz(t)(\mathbf{X})$ is the number of atoms in $\bar\Zz(t)$.
    From the construction of the
    truncated processes, we know that
    under $\lPn$, $T$ has an exponential distribution with rate $\lambda_{b_n} = \nu^{(b_n)}(\Pp\setminus\Pp_1)$.
    The point measure $\bar \Zz(T)$ has a countable number of atoms; let $(u^{(j)})_{j\ge 1}$
    be their labels, such that $u^{(1)} = \varnothing$ and $(u^{(j)})_{j\ge 2}$ is lexicographically ordered;
    in particular, this implies that $\Zz_{u^{(j)}}(T) \eqd \Zz_{\varnothing}(T-) + \log p_j$, where $\pp$
    is sampled from $\nu^{(b_n)}\rvert_{\Pp\setminus\Pp_1}/\lambda_{b_n}$.
    Furthermore, the translates $\Zz \circ \theta_{u^{(j)},T}$ are independent of each other and
    of $\FFs_T$, where
    we recall that $\FFs_t = \sigma( \bar\Zz(s), U_s; s\le t)$, for $t \ge 0$.
    Additionally, $(\Zz_\varnothing(s),s< T)$ is independent
    of $T$ and $\pp$, and has the law of a L\'evy process with Laplace exponent
    $\Psi^{(b_n)}$ killed at an independent exponential time of rate $\lambda_{b_n}$.
    
    All of these facts add up to the following computation, in which $F_j$ is a $\FFs_{t}$-measurable
    functional, $G_j$ is a measurable function of $\RR$ and $J$ is a measurable functional on the path space;
    and $u \in \tree$.
    Let $i$ be such that $u = u^{(i)}v$, with $i=1$ only if $u^{(j)}\not\prec u$ for all $j\ge 2$,
    and as a shorthand denote $\Delta\Zz_{u^{(j)}}(T) = \Zz_{u^{(j)}}(T) - \Zz_{\varnothing}(T-)$.
    \begin{eqnarr*}
      \eqnarrLHS{
        \lQn[ J(\Zz_\varnothing(s),s<T) \prod_{j\ge 1} F_j \circ \theta_{u^{(j)},T}
        \, G_j(\Delta \Zz_{u^{(j)}}(T)) ; U_{T+t} = u ]
      }
      &=& \lEn\mbiggl[ 
      e^{-T\kappa^{(b_n)}(\omega)} e^{\omega \Zz_{\varnothing}(T-)} J(\Zz_{\varnothing}(s),s<T)
      \Bigl( \prod_{j\ge 1} G_j(\Delta \Zz_{u^{(j)}}(T)) \Bigr) e^{\omega \Delta \Zz_{u^{(i)}}(T)} \\
      && \quad {} \cdot \lEn\mBigl[
      e^{-t\kappa^{(b_n)}(\omega)}\Bigl(\prod_{j\ge 1} F_j\circ \theta_{u^{(j)},T}\Bigr)
      \, e^{\omega(\Zz_u(T+t)-\Zz_{u^{(i)}}(T))}
      \mm\vert \FFs_T 
      \mr]
      \mr]
      \\
      &=&
      \int_0^\infty \dd r \, \exp(-\mu_{b_n}r) \sE[  J(\chi^{(b_n)}_{\omega}(s),s< r)] \\
      && \quad {} \cdot 
      \int_{\Pp\setminus\Pp_1} \nu^{(b_n)}(\dd \pp)\,  p_i^\omega \lEQn[F_i; U_t = v] \prod_{j\ge 1} G_j(\log p_j)  \prod_{j\ne i} \lEn[F_j],
    \end{eqnarr*}
    where in the last line we define
    \[ \mu_{b_n} \coloneqq \lambda_{b_n} + \kappa^{(b_n)}(\omega) - \Psi^{(b_n)}(\omega)
      = \int_{\Pp\setminus\Pp_1} \sum_{i\ge 1} p_i^\omega \, \nu^{(b_n)}(\dd \pp) ,
    \]
    and $\chi^{(b_n)}_{\omega}$, defined under a probability measure $\sP$, is a Lévy process whose
    Laplace exponent is an Esscher transform of
    $\Psi^{(b_n)}$, namely $\Ess_{\omega} \Psi^{(b_n)} \coloneqq \Psi^{(b_n)}(\omega + \cdot) - \Psi^{(b_n)}(\omega)$.

    We now turn to the process $(\bar\Yy,V)$, again under $\lQn$.
    We again define the branching time,
    \[ T = \inf\{ t\ge 0 : \card \bar \Yy(t) \ge 2 \}
      = \inf\{t \ge 0: N(\{t\}\times \mathcal{A}) > 0 \}
      ,
    \]
    where $\mathcal{A} =  (0,1) \times \NN \times (\Pp\setminus\Pp_1)$;
    that is, $T$ is the first time that a jump of $\xi$ is accompanied by immigration.
    We consider the quantity
    \begin{equation}\label{e:bn-functional}
        \lEQtn \mbigl[ J(\Yy_{\varnothing}(s), s < T) 
        \prod_{j\ge 1} F_j \circ \theta_{u^{(j)},T}
        \, G_j(\Delta\Yy_{u^{(j)}}(T)); V_{T+t} = u \mr],
    \end{equation}
    where $F_j, G_j, J$ are measurable functionals as above.
    
    Observe that, under $\lQn$, $N$ is a Poisson random measure with intensity
    $\eta^{(b_n)}$ as defined in \eqref{e:eta-n}.
    Now, by the definition of $T$ and standard properties of Poisson random
    measures \cite[\S O.5]{Ber-Levy}, we know that
    $T$ has an exponential distribution with rate
    \[ \int_{[0,1]\times\mathcal{A}} \eta^{(b_n)}(\dd s,\dd y, \dd i, \dd \pp)
      = \int_{\Pp\setminus\Pp_1} \sum_{i\ge 1} p_i^\omega \, \nu^{(b_n)}(\dd \pp)
      \eqqcolon \mu_{b_n} .
    \]
    In fact, we can say more: the restriction $N\rvert_{[0,T)\times (0,1)\times\NN\times \Pp}$
    has same law as the restriction
    $N\rvert_{[0,\tau)\times\mathcal{A}^c}$,
    where $\tau$ is an exponentially-distributed random variable with rate $\mu_{b_n}$
    which is independent of $N$,
    and $\mathcal{A}^c = (0,1)\times\NN\times\Pp_1$.
    This has implications for the process $(\Yy_{\varnothing}(s),s<T)$
    which, importantly, is the same as the spine process $\xi$ on the 
    time interval in question;
    it remains a L\'evy process with Gaussian coefficient $\sigma$,
    but has two changes: first, it is killed independently at rate $\mu_{b_n}$.
    Second, the law of its jump measure, which we recall is the pushforward of
    $N(\dd s, \dd y, \NN,\Pp)$ by the map $(s,y)\mapsto (s,\log y)$,
    is altered because the law of $N$ is altered.
    Working with the L\'evy--It\^o decomposition, we see that $(\Yy_{\varnothing}(s),s<T)$
    has Laplace exponent given by
    \begin{eqnarr*}
      \eqnarrLHS{\frac{1}{2}\sigma^2 q^2 + c_{b_n,\omega}q + \int_{[0,1]\times(0,1)\times\{1\}\times\Pp_1} \mbigl( y^q-1-ql(y) \mr) \, \eta^{(b_n)}(\dd s,\dd y, \dd i, \dd \pp) - \mu_{b_n}}
      &=& \frac{1}{2}\sigma^2 q^2 + c_{b_n,\omega}q + \int_{\Pp_1} \mbigl( p_1^q-1-ql(p_1) \mr) p_1^\omega\, \nu^{(b_n)}(\dd\pp) - \mu_{b_n}
      \\
      &=& \Ess_\omega\Psi^{(b_n)}(q) - \mu_{b_n}, \qquad q \ge 0,
      \IEEEyesnumber\label{e:fb-n:le}
    \end{eqnarr*}
    where
    \begin{eqnarr*}
      c_{b_n,\omega}
      &=&
      a + \omega\sigma^2 + \int_{\Pp} (1-p_1) \, \nu^{(b_n)}(\dd\pp)
      + \int_{\Pp_1} p_1^\omega l(p_1)\, \nu^{(b_n)}(\dd\pp) .
    \end{eqnarr*}
    Note that the centre $c_{b_n,\omega}$ differs from the centre of $\xi$ due to
    the change in compensation of the small jumps.
    It follows that $(\Yy_{\varnothing}(s),s<T)$ 
    has the law of $\chi^{(b_n)}_{\omega}$ killed at an independent
    exponential time with rate $\mu_{b_n}$.
    
    Considering the particles born at time $T$, define the children $(u^{(j)})_{j\ge 1}$
    of $\Yy_\varnothing$ as for the previous part of the proof,
    and assume that again $u = u^{(i)}v$, with the convention that $i = 1$ only if $u^{(j)}\not\prec u$
    for all $j\ge 2$.
    Using again the properties of Poisson random measures, the atom $(T,y,k,\pp)$
    of $N$ appearing at time $T$ is such
    that $(y,k,\pp)$ has distribution $\frac{\eta^{(b_n)}([0,1],\cdot)\rvert_{\mathcal{A}}}{\mu_{b_n}}$,
    and we are further restricted in \eqref{e:bn-functional}
    to the event $V_{T+t}=u$, which implies that here
    we are restricted to the event $\{k=i\}$.
    Finally, from the construction of the decorated spine process,
    we know that each child $u^{(j)}$ is initially positioned at $\Yy_\varnothing(T-)+\log p_j$,
    and
    that the translate $\Yy \circ \theta_{u^{(i)},T}$ has the law of $\Yy$ under $\lQtn$,
    while the translates $\Yy \circ \theta_{u^{(j)},T}$ are independent of one another
    and have the law of $\Zz$ under $\lP^{(b_n)}$.
    
    The discussion above essentially proves the required decomposition, but for clarity we
    provide the following calculation, in which $J, F_j, G_j$ are measurable functionals as above.
    \begin{eqnarr*}
      \eqnarrLHS{
        \lEQtn \mbigl[ J(\Yy_{\varnothing}(s), s < T) 
        \prod_{j\ge 1} F_j \circ \theta_{u^{(j)},T} G_j(\Delta\Yy_{u^{(j)}}(T)); V_{T+t} = u \mr]
      }
      &=&
      \int_0^{\infty} \dd r \,
      \mu_{b_n} \exp(-\mu_{b_n}r) \sE[ J(\chi^{(b_n)}_{\omega}(s), s < r) ] \\
      &&  {} \cdot \int_{(0,1)\times(\Pp\setminus \Pp_1)}
      \frac{1}{\mu_{b_n}}\eta^{(b_n)}([0,1],\dd y, \{i\}, \dd \pp) \,
      G_i(\log y) \lEQtn[F_i; V_t = v] \prod_{j\ne i} \lEn[F_j] G_j(\log p_j)
      \\
      &=& \int_0^{\infty} \dd r \, \exp(-\mu_{b_n}r) \sE[ J(\chi^{(b_n)}_{\omega}(s), s < r) ]
      \\
      && {} \cdot
      \int_{\Pp\setminus\Pp_1} \nu^{(b_n)}(\dd\pp) \, p_i^\omega \prod_{j\ge 1} G_j(\log p_j)
      \lEQtn[F_i; V_t = v] \prod_{j\ne i} \lEn[F_j].
    \end{eqnarr*}
    This completes the proof.
  \end{proof}

Having established the result for these truncated processes, we need to
remove the truncation, and this proves the following
theorem on the spine decomposition,
which is our main result.

  \begin{thrm}[Spine decomposition]\label{t:mto}
    Under $\lQ$, $(\bar \Zz(t),U_t)_{t\ge 0}$
    is equal in law to $(\bar \Yy(t),V_t)_{t\ge 0}$.
  \end{thrm}
  \begin{proof}
   Since the processes $(\bar{\Zz},U)$ and $(\bar{\Yy},V)$ are both c\`adl\`ag,
    it is sufficient to prove that they have the same finite-dimensional distributions. 
  For simplicity, we shall only establish the convergence for one-dimensional; similar but more cumbersome arguments hold for multi-dimensional case.
  
  Fix $t\ge 0$. 
    For the measure
    $\lQn$, \eqref{e:Qn} implies 
    \[
     \lEQn[ F(\bar \Zz(t)) \Indic{U_t=u}] 
      = e^{-t\kappa^{(b_n)}(\omega)} \lE[ F(\bar \Zz^{(b_n)}(t)) e^{\omega \Zz^{(b_n)}_u(t)}]
    \]
    for continuous bounded $F$.
    Under $\lP$, $\bar{\Zz}^{(b_n)}(t)\to \bar{\Zz}(t)$ weakly on $\Mmp(\mathbf{X})$.
    Furthermore, for every $\omega\in\dom\kappa$,
    $\kappa^{(b_n)}(\omega)\upto\kappa(\omega)$.
    Hence, certainly
    the distribution of $(\bar{\Zz}(t),U_t)$
    under
    $\lQn$  converges weakly to the distribution of $(\bar{\Zz}(t),U_t)$
    under $\lQ$.
    
    We now address the convergence
    of the law of $(\bar\Yy(t),V_t)$.
    Consider first the process $\bar \Yy^{(b_n)}$, which was defined in \autoref{s:forward},
    using the notation $A_{b_n}$ and $\tau_{b_n}$.
    We may consider
    the joint process $(\bar\Yy^{(b_n)},V^{(b_n)})$, by adjoining a `cemetery' element $\partial$ 
    to the collection of labels, and defining
    $V_t^{(b_n)} = V_t\Indic{t < \tau_{b_n}} + \partial \Indic{t \ge \tau_{b_n}}$; thus,
    $V_t^{(b_n)} = \partial$ indicates that the distinguished line of descent has been
    killed before time $t$ in the process $\bar\Yy^{(b_n)}$.
    
    
    We will start by showing that, for $F$ a continuous bounded functional on $\Mmp(\mathbf{X})$
    and $u\in\tree$,
    \begin{eqnarr*}
      \lEQt[ F(\bar \Yy^{(b_n)}(t)) \Indic{V_t^{(b_n)}=u} \mid V_t^{(b_n)} \ne \partial]
      &=& \lEQtn[ F(\bar \Yy(t)) \Indic{V_t = u}]  \IEEEyesnumber \label{e:fb-inter1} \\
      &=& \lEQn[ F(\bar\Zz(t)) \Indic{U_t = u}]. 
    \end{eqnarr*}
    The second equality is an immediate corollary of
    \autoref{l:fb-n}, so we have only to prove the first equality.
    
    The conditioning on the left-hand side of \eqref{e:fb-inter1} is the same as conditioning
    on the event $\{ t < \tau_{b_n}\}$, where $\tau_{b_n}$ is the hitting time of the set $A_{b_n}$ for the
    Poisson random measure $N$.
    We notice that, given $\{t < \tau_{b_n}\}$, we have the equality
    \begin{equation*}
      \Yy^{(b_n)}(t) = \delta_{\xi(t)} 
      + \int_{[0,t]\times (0,1)\times\NN\times\Pp} {N}_{b_n}(\dd s,\dd y, \dd i, \dd \pp)
      \sum_{j\ne i} \mbigl[ (\Zz^{[s,j]})^{(b_n)} + \xi(s-) + \log p_j \mr],
    \end{equation*}
    where
    $N_{b_n}$ is defined in \eqref{e:Nb}.
    
    Using standard properties of Poisson random measures
    \cite[\S O.5]{Ber-Levy},
    we see that, under $\lQt(\cdot \mid t < \tau_{b_n})$,
    $N$ has the same law as
    its restriction
    $N\rvert_{[0,\infty)\times A_{b_n}^c}$.
    Thus, under $\lQt(\cdot\mid t<\tau_{b_n})$,
    $N_{b_n}$ is a Poisson random measure whose intensity
    is the measure $\eta^{(b_n)}$ given in \eqref{e:eta-n}.
    
    The change in the law of the measure $N$ which is induced by this
    conditioning causes a corresponding change to the jump measure of $\xi$,
    which we again recall is the pushforward of $N(\dd s, \dd y, \NN, \Pp)$ under
    $(s,y) \mapsto (s,\log y)$.
    Using the L\'evy--It\^o decomposition much as in the proof of \eqref{e:fb-n:le},
    we may show
    that under $\lQt(\cdot \mid t < \tau_{b_n})$,
    $(\xi,N_{b_n})$ has the same law as $(\xi,N)$ does under $\lQtn$.
    Finally, $\bar{\Yy}^{(b_n)}$ is measurable with respect to
    $N_{b_n}$ and $\xi$, and the same is true of $V_t^{(b_n)}$
    on the event $\{t<\tau_{b_n}\}$.
    This completes the proof of \eqref{e:fb-inter1}.
    
    We now need to take $n\to\infty$. The right-hand side of \eqref{e:fb-inter1}
    converges to $\lEQ[F(\bar\Zz(t)) \Indic{U_t = u}]$, as discussed 
    at the beginning of the proof.
    The left-hand side of \eqref{e:fb-inter1} is equal to
    \begin{equation}\relax\label{e:fb-inter2}
      \frac{\lEQt\mbigl[ F(\bar{\Yy}^{(b_n)}(t)) \Indic{V^{(b_n)}_t = u}\mr]}%
      {\lQt(V_t^{(b_n)}\ne \partial)}.
    \end{equation}
    For every $t\ge 0$ and every realisation of the process,
    $\{V_t^{(b_n)}\ne \partial\}$ holds for
    large enough $n$; moreover,
    by \autoref{l:cv-Ybn} we have $\bar{\Yy}^{(b_n)}(t)\to \bar{\Yy}(t)$ in probability,
    and hence (extracting a subsequence if necessary) also almost surely.
    It follows from the dominated convergence
    theorem that \eqref{e:fb-inter2} converges to
    $\lQt[F(\bar \Yy(t)) \Indic{V_t = u}]$. This completes the proof.
  \end{proof}
  
  \begin{rem}
  \fakephantomsection\label{r:mto}\label{r:BM}
    \begin{enumerate}
      \item\label{i:mto}
      We stress that a version of this theorem has been proved, by \citet{BBCK-maps},
      for the case of binary branching ($\nu(\dd\pp)$ being supported by those $\pp$
      such that $p_3 = 0$) under the condition that $\kappa(\omega) = 0$,
      though their description of the decomposition differs somewhat from ours due to their
      view of the genealogy.
      
      \item \label{i:BM}
      After the initial appearance of this work,
      \citet[Lemma~2.3]{BM-mcbLp} gave a version of this theorem, 
      albeit without a full genealogy in terms of labels,
      for branching L\'evy processes \cite{BM-bLp} in which upward jumps of the particle locations are permitted.
  		Using methods in this current work, it should be possible to establish the 
  		spine decomposition for the labelled process under the same assumptions
  		as \cite{BM-mcbLp};
  		recall also \autoref{r:misc}. 
    \end{enumerate}
  \end{rem}
  
  The theorem above establishes a `full many-to-one theorem' in
  	the language of \cite{HH-spine}, and we have as an immediate corollary the following useful expression for certain functionals of $\Zz$: 
  \begin{cor}[Many-to-one formula]\label{c:mto}
    For $t\ge 0$ and any non-negative $\mathcal{F}_t$-measurable functional $f$,
    \[ \lE\mBigl[\sum_{u\in \tree_t} f(\Zz_u(s),s\le t) \mr] = e^{t\kappa(\omega)} \sE[ f(\xi(s), s\le t) e^{-\omega\xi(t)} ]  \]
    where $\xi$ under the measure $\sP$ is a L\'evy process with Laplace exponent $\Ess_\omega\kappa$.
  \end{cor}
  We also point out the following consequence for the process $\bar\Yy$. Recall
  that $(\bar{\Yy},V)$ is Markov; this result says the same is true even if
  we forget $V$.
  \begin{cor}
    $\bar \Yy$ is a Markov process under $\lQ$.
  \end{cor}
  \begin{proof}
   $\bar\Zz$ is defined (without the distinguished
   particle $U$)
   by a change of measure of a Markov process with respect to the martingale
   $\mg{\omega}{\cdot}$, and is therefore a Markov process in its own right.
   $\bar\Yy$ is equal in distribution to $\bar\Zz$ under $\lQ$, and this
   completes the proof.
  \end{proof}

 \section{The derivative martingale}
  \label{s:dm} 
For every $\omega\in (\dom \kappa)^{\circ}$, the interior of $\dom \kappa$,  let $\partial W(\omega,\cdot)$ denote the derivative martingale, given by
  \[ \partial W(\omega,t) = e^{-t\kappa(\omega)}
    \sum_{u\in \tree_t} \big(\Zz_u(t) - t\kappa'(\omega)\big) e^{\omega \Zz_u(t)}, \qquad t \ge 0 . \]
  Our purpose is to study the asymptotic properties of this martingale.


Before stating our main result, Theorem~\ref{t:DerMart}, let us first distinguish two regimes of $\omega$.
By the convexity of $\kappa$, we observe that the function $q\mapsto q \kappa'(q) -\kappa(q)$ is increasing on $(\dom \kappa)^{\circ}$ and has at most one sign change. From now on, we assume that  
 \begin{equation}\label{H}
   \text{ there exists (a unique) } \ws>0, \text{ such that } \ws \in (\dom \kappa)^{\circ} ~\text{and}~\ws\kappa'(\ws) -\kappa(\ws)=0. \tag*{(H)}
  \end{equation}
The value $\ws$ has proved to be critical for the study of the uniform integrability of  the exponential additive martingale $\MA(\omega, \cdot)$; see \cite{Dadoun:agf, BM-mcbLp}. 
We point out that the assumption \ref{H} entails
that $\kappa(0)\in (0,+\infty]$, so the \emph{non-extinction}
event has strictly positive probability:
 \[ \lP\mbigl(\#\Zz(t) >0 \text{ for all }t\ge 0 \mr)\in (0,1],\]
where $\#\Zz(t) \coloneqq \Zz(t)(\RR)$ denotes the number of atoms at time $t$.
Write $\lP^*$ for the probability measure $\lP$ conditional on non-extinction.
We recall our standing assumption $\nu(\{\mathbf 0\}) = 0$,
implying that particles are never killed;
this in fact implies that non-extinction occurs $\lP$-almost surely,
and so in fact $\lP^* = \lP$ for us. However, we retain the notation $\lP^*$
in order to make clear how our results would look without our assumption.

\skippar
We now state the main result of this section. 
\begin{thrm}\label{t:DerMart}  Suppose that \ref{H} holds. 
	\begin{enumerate}
		\item Let $\omega \geq \ws$, then the derivative martingale $\MD(\omega, t)$ converges $\lP$-almost surely to a finite non-positive limit $\MD(\omega,\infty)$ as $t\to \infty$. 
		\item Let $\omega > \ws$, then $\MD(\omega,\infty) = 0$ holds $\lP$-almost surely, 
		\item For $\omega=\ws$, there is $\lE[\MD(\ws,\infty)] = -\infty$, and $\MD(\ws,\infty) <0$ holds $\lP^*$-almost surely.
	\end{enumerate}
\end{thrm}

In \cite[Corollary 2.10(b)]{Dadoun:agf}, Dadoun has shown
the $\lP^*$-almost sure negativity of the random variable $\partial W(\bar{\omega},\infty)$,
identified there as the almost sure limit of the discrete martingale
$(\partial W(\bar{\omega},n), n = 0,1,\dotsc)$ Our theorem improves upon
\cite{Dadoun:agf} by proving convergence of the continuous-time martingale
and finding the expected value of the limit random variable.
Furthermore, we do not require condition (2.7) of \cite{Dadoun:agf}.

The limit $\partial W(\bar{\omega},\infty)$ has an intimate connection with the asymptotic behaviour of the largest fragment and Seneta--Heyde norming for $\MA(\ws, \cdot)$; see \cite[Corollary~2.10 and Remark~2.11]{Dadoun:agf}. 
Analogues of \autoref{t:DerMart} were proved for multitype branching random walks by \citet{BK-mc}, for branching Brownian motion by \citet{Kyp-FKPP} and for pure fragmentation processes by \citet{BR-disc}. A thorough exposition of the theory for branching random walks is given in the monograph of \citet{Shi-BRW}. 

The common approach of the works described above is 
a technique based upon stopping particles moving at a certain speed,
and we stress that the spine decomposition plays a central role in these arguments.
Our proof, which is primarily modelled
on that of \citet{BR-disc}, is postponed to \autoref{s:proofD}; in the coming \autoref{s:sm}, we prepare for it by investigating a related family of martingales.

  \begin{rem}
  		   For generic branching L\'evy processes  \cite{BM-bLp} in which upward jumps of the particle locations are permitted, the same arguments apply to prove (i) and (ii) of \autoref{t:DerMart}, but not (iii).  
  		    For $\omega=\bar\omega$, we expect that an additional assumption in terms of the dislocation measure $\nu$ is needed to make the limit non-trivial. 
  		  In the case of branching random walks \cite{Chen-dm} and branching Brownian motion \cite{RY-dm}, optimal moment conditions have been found, and the martingale limits are proven to be zero when these conditions do not hold. 
  \end{rem}


\subsection{The stopped martingales}\label{s:sm}
In this subsection we fix $a>0$ and $\omega \in (\dom \kappa)^{\circ}$, and define a process
\[
  \MD_a(\omega,t):= \sum_{u\in \tree_t} \mbigl( a +t \kappa'(\omega) -\Zz_u(t)  \mr) e^{-t \kappa(\omega) + \omega\Zz_u(t)}\Indic{ a+r \kappa'(\omega) - \Zz_{ \Anc(r,u)}(r)>0 \text{ for } r\le t},\quad t\ge 0,
\]
where $\Anc(r,u)$ denotes the ancestor of $u$ at time $r$ as in \autoref{s:branching}. It is clear that  $\MD_a(\omega,t)$ is always
non-negative.
 We use this to define a new measure on $\FF_\infty$ by
  \[ \lQD (A) := \frac{1}{a}\lE[\partial W_a(\omega,t) \Ind_A ] , \qquad t \ge 0, A \in \FF_t, \]
  and extend it to $\hat\FF_\infty$ by
  \begin{multline}\label{e:lQD} 
  \lQD (A; U_t = u) \\
  \coloneqq \frac{1}{a}\lE\mbigl[\mbigl(a + t\kappa'(\omega) - \Zz_u(t)\mr) e^{-t \kappa(\omega)+\omega \Zz_u(t)}
    \Indic{a+s\kappa'(\omega)-\Zz_{\Anc(r;u)}(r) > 0 \text{ for }r\le t} \Ind_A  \mr]. \end{multline}

  To justify that the measure $\lQD$ is well-defined and does not depend on the choice of $t$, we consider the interpretation of $\lQD$ as having a density with respect to $\lQ$ on $\hat\FF_\infty$.
  Recall that under $\lQ$, we have a process $\bar\Zz$ together with a spine label $U$, and the spine $(\Zz_{U_t}(t), t\geq 0)$ is a L\'evy process with Laplace exponent $\Ess_\omega\kappa$. Write 
\begin{equation}\label{e:lambda}
   \lambda(t) := a + t\kappa'(\omega) - \Zz_{U_t}(t), \qquad t\ge 0, 
\end{equation}
  then it follows that $\lambda$ under $\lQ$ is a L\'evy process
  with respect
  to the filtration $(\hat \FF_t)_{t\ge 0}$, started at $a$. The process $\lambda$
  is spectrally positive, in the sense that it has only positive jumps, and it
  has Laplace exponent $\kappa'(\omega)q - \Ess_\omega\kappa(q)$,
  meaning that $\lEQ[e^{-q(\lambda(t)-a)}] = e^{-t(\kappa'(\omega)q-\Ess_\omega\kappa(q))}$
  for $q\ge 0$. (This is a slight change in notation compared to
  \eqref{e:lk}, but it follows the usual convention for the Laplace exponent of a
  spectrally positive process.)
  In particular, $\lEQ [ \lambda(t)] = a$ for every $t\geq 0$, which implies that $\lambda$ is a $\lQ$-martingale with respect to $(\hat\FF_t)_{t\ge 0}$.
  Let
  \[ \zeta = \inf\{ t\ge 0: \lambda(t) < 0 \}, \] 
  then it follows from \autoref{c:mto} that
   \begin{equation}\label{e:cm-lambda}
	\lQD (A) =a^{-1}\lEQ\mBigl[ \lambda(t) \Indic{t < \zeta} \Ind_{A} \mr] , \qquad t \ge 0, A \in \hat\FF_t. \end{equation}
  Using the fact that the stopped martingale $(\lambda(t\wedge \zeta)= \lambda(t) \Indic{t < \zeta}, t\geq 0) $ remains a $\lQ$-martingale (see \cite[Corollary II.3.6]{RY-cmbm}), we justify the previous definition of $\lQD$ as a consistent change of measure.
  As a consequence, $\partial W_a (\omega, \cdot)$ is a non-negative $\lP$-martingale, and therefore converges $\lP$-almost surely to a limit $\MD_a(\omega, \infty)$ as $t\to \infty$.

  \skippar
  The main object of this subsection is to establish the following result, which will be crucially used in the proof of \autoref{t:DerMart}.
\begin{prop}\label{l:MDa}Suppose that \ref{H} holds.
  \begin{enumerate}
  \item For $\omega> \ws$, we have $\lP$-almost surely $\MD_a(\omega,\infty)= 0$. 
  \item For $\omega= \ws$, the martingale $\MD_a(\ws,t)$ converges to $\MD_a(\ws,\infty)$ in $L^1(\lP)$. 
   \end{enumerate}
  \end{prop}

  To prove \autoref{l:MDa}, the key idea is to use the `forward' construction (\autoref{d:deco2} and \autoref{t:mto}) of $(\bar\Zz,U)$ under $\lQ$, as a L\'evy process $\xi$ with Laplace exponent $\Ess_\omega\kappa$ whose jumps are decorated with independent branching L\'evy processes with law $\lP$, each
  positioned according to the atoms of a random measure $N$.
  By a slight abuse of notation, the measure $N$ under $\lQ$
  can be seen as an
      integer-valued random measure on $[0,\infty)\times E$, with $E = (0,1)\times\NN\times\Pp$,
      and its support is a random set having the form
      $\{ (s,(e^{-\Delta \xi(s)},i_s,\pp_s)) : \Delta\xi(s)\ne 0\}$.
      Further, $N$ is Poisson with the (non-random) intensity measure $\eta$. 
 Since $\lQD$ is absolutely continuous with respect to $\lQ$
 on every $\hat\FF_t$, the process under $\lQD$ has the same structure; however, the laws of the process $\xi$ and the random measure $N$ may be different.
 
  The following pair of lemmas provides 
  more detail on the discussion above;
  we refer to \citet[\S II.1]{JS-limit} for a thorough discussion of
  random measures, and in particular the
  notion of the predictable compensator of a random measure.
  Note that hereafter, when we say \emph{predictable},
  we will always mean predictable with respect to the filtration
  $(\hat\FF_t)_{t\ge 0}$.

  \begin{lem}\label{l:clambda}
    Under $\lQD$, the process $\lambda$ defined as in \eqref{e:lambda} is a spectrally positive
    L\'evy process starting from initial value $a > 0$ with 
    Laplace exponent $q \mapsto \kappa'(\omega)q - \Ess_\omega\kappa(q)$, conditioned to
    be positive in the sense of \cite{Cha-cond,CD-cond}.
    In particular, we have that $\inf_{t\ge 0} \lambda(t) > 0$, $\lQD$-almost surely.
    \begin{proof}
      Recall that $\lambda$ is a (unconditioned) L\'evy process with the given Laplace exponent under $\lQ$. 
      In the work of \citet{CD-cond}, it is shown that conditioning the L\'evy process $\lambda$
    	 to remain
    	 positive is equivalent to performing a martingale change of measure with respect to the
    	 martingale $U_-(\lambda(t))\Indic{t<\zeta}$, where $U_-$ is the potential function
    	 of the downward ladder height subordinator.
    	 Since $\lambda$ has no negative jumps and
    	 has constant mean $a$, it follows that
    	 $U_-(x) = x \Indic{x>0}$ (see \cite[\S 6.5.2]{Kyp2} for
    	 the analogous case of processes with no positive jumps.) 
    	 Therefore, conditioning $\lambda$ to remain positive gives rise to
    	 $\lQD$ as the conditioned measure.
    	 
    	 This completes the characterisation of $\lambda$ under $\lQD$.
    	 Finally, since $\lambda$ under $\lQ$ is a centred L\'evy process with only positive jumps, the fact that the overall infimum of $\lambda$ under $\lQD$ is positive is implied by \cite[Theorem 1(a)]{CD-cond}.
   \end{proof}
  \end{lem}

  \begin{lem}\label{l:cN}
    The predictable compensator of the random measure $N$ under $\lQD$ is given by
    \[ \eta'(\dd s, \dd y, \dd i, \dd \pp) \coloneqq 
      \frac{\lambda(s-) - \log y}{\lambda(s-)}
      \, \eta(\dd s, \dd y, \dd i, \dd \pp) .
    \]
   \begin{proof}

      We first point out that $\zeta$ is predictable:
      since $\lambda$ is a spectrally positive L\'evy process under $\lQ$, it can only pass
      below $0$ continuously. Thus, defining
      $T_n = \inf\{ t \ge 0: \lambda(t) < 1/n\} < \zeta$,
      we have that $\zeta = \sup_{n} T_n$, which implies in particular that $\zeta$ is predictable (by \cite[Theorem I.2.15(a)]{JS-limit}.) 
      
      Now, since $N$ is Poisson under $\lQ$, its compensator under $\lQ$ is the (non-random) intensity measure $\eta$, and moreover
      the density process for the change of measure is
      \[
        \frac{\dd \lQD}{\dd \lQ}\bigg|_{\hat\FF_s} 
        = a^{-1}\lambda(s) \Indic{s<\zeta}.
      \] 
      For any predictable random function $(s,(y,i,\pp)) \mapsto U_s(y,i,\pp)$,
      we have that
      \begin{multline*}
        \lEQ \mbiggl[
        \int_{[0,\infty)\times E} \frac{\lambda(s)}{\lambda(s-)} \Indic{s<\zeta}
        U_s(y,i,\pp) \, N(\dd s, \dd y,\dd i,\dd \pp) \mr] \\
        =
        \lEQ \mbiggl[
        \int_{[0,\infty)\times E} \frac{\lambda(s-)-\log y}{\lambda(s-)} \Indic{s<\zeta}
        U_s(y,i,\pp) \, N(\dd s, \dd y,\dd i,\dd \pp) \mr],
      \end{multline*}
      and the random function
      $(s,(y,i,\pp)) \mapsto \frac{\lambda(s-)-\log y}{\lambda(s-)} \Indic{s<\zeta}$
      is predictable. Having made these observations, the result
      follows by the Girsanov theorem for random measures
      \cite[Theorem III.3.17(b)]{JS-limit}.
      Note that under $\lQD$, $\zeta = \infty$ by \autoref{l:clambda}.
   \end{proof}
  \end{lem}

We now need one final technical result to prepare for the
main proposition in this section.
  \begin{lem}\label{l:lambda-int}
    For every $p>0$,
    \[ \lQD \mbiggl[ \int_{0}^\infty \mBigl(\lambda(r)+1+\frac{1}{\lambda(r)}\mr) e^{-p \lambda(r)} \, \dd r \mr] < \infty . \]
    \begin{proof}
      Let $V(a, \dd y) = \lQD \mbigl[ \int_0^\infty \Indic{\lambda(r)\in \dd y} \, \dd r \mr]$,
      and $U^\dag(a,\dd y) = \lQ \mbigl[ \int_0^{\zeta} \Indic{\lambda(r)\in \dd y}\, \dd r \mr]$.
      The former is the potential of a Lévy process conditioned to stay positive, and the latter
      is that of a Lévy process killed upon going below the level $0$. Due to the $h$-transform
      connecting their semigroups, they are related by the formula
      \[ V(a,\dd y) = \frac{y}{a}U^\dag(a,\dd y). \]
      Moreover, by \cite[Corollary 8.8]{Kyp2}, we have
      \[ U^\dag(a,\dd y) = (\WW(y)-\WW(y-a)) \, \dd y , \qquad y \ge 0, \]
      where $\WW$ is the scale function of the spectrally negative Lévy process $-\lambda$,
      with the convention that $\WW(x)=0$ for $x< 0$,
      and $k >0$ is a constant.
      
      Thus, we have that
      \begin{eqnarr*}
        \lQD \mbiggl[ \int_{0}^\infty \mBigl(\lambda(r)+1+\frac{1}{\lambda(r)}\mr) e^{-p \lambda(r)} \, \dd r \mr]
        &=& \int_{[0,\infty)} \mBigl(y+1+\frac{1}{y}\mr) e^{-py} \, V(a,\dd y) \\
        &=& k \int_0^\infty \frac{y}{a}\mBigl(y+1+\frac{1}{y}\mr) e^{-py} (\WW(y)-\WW(y-a)) \, \dd y.
      \end{eqnarr*}
      
      Finally, by \cite[equation (VII.4)]{Ber-Levy}
      and the renewal theorem \cite[Theorem III.21]{Ber-Levy}, we know that
      $\WW(y)-\WW(y-a) \to ac/{m_+}$ as $y \to \infty$, where $m_+$ is the mean of the ascending
      ladder height process of $\lambda$ and $c$ is a meaningless constant.
      This implies that the integral above converges at $\infty$.
      We then note that the integrand is equivalent to $ k a^{-1} (y^2+y+1) \WW(0)$ as $y\to 0$,
      and this completes the proof.      
    \end{proof}

  \end{lem}
  \begin{rem}
    In \cite{BR-disc}, the result
    \[ 
      \inf_{t\ge 0} \lambda(t) > 0 
      \qquad \text{and}\qquad 
      \lim_{t\to \infty} 
      \frac{\log \lambda(t)}{\log t} 
      = \frac{1}{2} \qquad\quad \text{$\lQD$-almost surely},
    \]
    stated in \cite[equation (21)]{BR-disc},
    is used. This would suffice for our purposes also.
    However, since the proof of \autoref{l:lambda-int} is not very long,
    we offer it for the sake of completeness.
  \end{rem}

\skippar
We are now in a position to prove \autoref{l:MDa}.
\begin{proof} [Proof of \autoref{l:MDa}] 
   (i)~By a fundamental result in measure theory (see e.g. \cite[Corollary~1]{Ath-change}), it suffices to prove that 
	\[
        \MD_a(\omega,\infty)  = \infty, \quad \lQD\text{-a.s.}
	\]
        For $\lambda$ defined by \eqref{e:lambda} and every $t\geq 0$, it is clear that 
	\[\MD_a(\omega,t) \geq \lambda(t) \exp \mbigl(a\omega+\omega \kappa'(\omega) t- \kappa(\omega)t + \omega\lambda(t)\mr) .\]
	As $\omega > \ws$, there is $\kappa'(\omega)\omega > \kappa(\omega)$. Moreover, we know from \autoref{l:clambda} that $\inf_{t\ge 0} \lambda(t) > 0$. The claim follows as a consequence. 
	
\bigskip	
(ii)~Let us start with a useful estimate. Under assumption \ref{H}, we can choose $\epsilon>0$  small enough such that $\ws-\epsilon\in (\dom \kappa)^{\circ}$ and that $\ws-\epsilon>0$. 
Then there is 
  \begin{equation}\label{e:h-dm}
 \int_{\Pp} \mBigl( \sum_{i\ge 2} p_i^{\omega}  \mr) \mBigl(\log_+\mbigl(\sum_{i\ge 2} p_i^{\omega}\mr) \mr)^{\rho} \, \nu(\dd \pp) <\infty, \qquad \text{for all}~ \omega\in [\ws-\epsilon, \ws], \rho\in[1,2]. 
\end{equation}
To prove \eqref{e:h-dm}, for any $\omega\in [\ws-\epsilon, \ws]$ and $\rho\in[1,2]$, we next choose $\gamma >0$ small enough, such that 
$\omega + \rho \gamma (\omega -1) \in \dom \kappa$. 
 Using the inequality $\log_+ y \le \gamma^{-1} y^{\gamma}$ for all $y\ge 0$ and Jensen's inequality, we have  
\[
\mBigl( \sum_{i\ge 2} p_i^{\omega}  \mr) \mBigl(\log_+\mbigl(\sum_{i\ge 2} p_i^{\omega}\mr) \mr)^{\rho} \le \gamma^{-2} \Big( \sum_{i\ge 2} p_i p_i^{\omega-1} \Big)^{1+\rho \gamma} 
 \le \gamma^{-2} \sum_{i\ge 2} p_i^{\omega+\rho \gamma(\omega-1)}.
\]
Since $\omega + \rho \gamma (\omega -1) \in \dom \kappa$,  by \eqref{eq:dom} we deduce \eqref{e:h-dm}. 

\skippar
We now come back to the proof of the proposition. 
		 By \cite[Lemma 4.2]{Shi-BRW} it suffices to show that 
      \begin{equation}\label{e:liminf}
	\liminf_{t\to \infty}  \lQDc \mbigl[ \MD_a(\ws,t) \mm\vert \GG_{\infty} \mr]<\infty, \quad \lQDc \text{-a.s.},
      \end{equation}
      where $\GG_{\infty}:= \sigma(\Zz_{U(t)}(t),U(t), t\geq 0) \subset \hat{\FF}_{\infty}$. 
Recall that $\lQDc$ is related to $\lQc$ via the change of measure \eqref{e:cm-lambda}, and that 
$\Zz$ under $\lQc$ can be described as a decorated spine process as in \autoref{d:decorated-spine}. With notation therein, we claim that 
\begin{equation}\label{e:Q=P}
\lQDc\mbigl[\MD_a(\ws,t)\mm\vert \GG_{\infty} \mr]=S_t, \qquad \lQDc-a.s. , 
\end{equation}
where, with $\lambda(t)= a - \xi(t) + t\kappa'(\ws)$ and $\zeta =  \inf\{ t\ge 0: \lambda(t) < 0 \}$, 
\begin{multline}
S_t:= \lambda(t) e^{a\ws  - \ws\lambda(t)}\Indic{t<\zeta} \\
+ \int_{[0,t]\times(0,1)\times\NN\times\Pp}
e^{a\ws  - \ws\lambda(r-)} \Indic{r<\zeta}  \sum_{j \ne i} (\lambda(r-)- \log p_j)  p_j^{\bar{\omega}} \, N(\dd r,\dd y, \dd i, \dd \pp).  \label{e:MD-g}
\end{multline}

We postpone for a moment the proof of \eqref{e:Q=P} and turn our attention to $S_t$. 
Let 
\[
X:= \sum_{i\ge 2} p_i^{\ws}, \quad \text{and}\quad \tilde{X} := \sum_{i\ge 2} p_i^{\ws-\epsilon}. 
\]
Fix $\theta\in (0,\ws-\epsilon)$, let 
\begin{eqnarr*}
A_t&:=& \lambda(t) e^{a\ws  - \ws\lambda(t)}\Indic{t<\zeta},\\
B_t&:=&  \int_{[0,t]\times(0,1)\times\NN\times\Pp}
e^{a\ws  - \ws\lambda(r-)} \Indic{e^{\theta\lambda(r-)}>X} \Indic{e^{\theta\lambda(r-)}> \tilde{X}}
 \sum_{j \ne i} (\lambda(r-)- \log p_j)  p_j^{\bar{\omega}} \, N(\dd r,\dd y, \dd i, \dd \pp),\\
D_t&:=&\int_{[0,t]\times(0,1)\times\NN\times\Pp}
e^{a\ws  - \ws\lambda(r-)} \Indic{e^{\theta\lambda(r-)}\le X} \Indic{i\ne 1}
\sum_{j \ne i} (\lambda(r-)- \log p_j)  p_j^{\bar{\omega}} \, N(\dd r,\dd y, \dd i, \dd \pp),\\
D'_t&:=&\int_{[0,t]\times(0,1)\times\NN\times\Pp}
e^{a\ws  - \ws\lambda(r-)}  \Indic{e^{\theta\lambda(r-)}\le \tilde{X}} \Indic{i\ne 1}
\sum_{j \ne i} (\lambda(r-)- \log p_j)  p_j^{\bar{\omega}} \, N(\dd r,\dd y, \dd i, \dd \pp).\\
E_t&:=&\int_{[0,t]\times(0,1)\times\NN\times\Pp}
e^{a\ws  - \ws\lambda(r-)}  \Indic{i= 1}
\sum_{j \ne i} (\lambda(r-)- \log p_j)  p_j^{\bar{\omega}} \, N(\dd r,\dd y, \dd i, \dd \pp).
\end{eqnarr*}
Then it is clear that 
\[
S_t \le A_t+B_t+D_t + D'_t + E_t. 
\]
We shall study the asymptotics of the five terms separately. 

      \skippar
Let us start with $B_t$. 
Using the compensator of $N$ under $\lQDc$ given in \autoref{l:cN} and \autoref{d:eta}, we deduce that
\[
\lQDc \mbigl[  B_t \mr]
=   \int_{[0,t]} \lQDc \mBigl[e^{a\ws  - \ws\lambda(r)} 
 \mbigl(C_1 \lambda(r)+ C_0 +C_{-1}\lambda(r)^{-1}\mr)\mr]\, \dd r,
\]
      where $C_1$, $C_0$ and $C_{-1}$ are given by
      \[
      C_1 := \int_{\Pp\setminus \Pp_1} \Indic{e^{\theta\lambda(r)}>X} \Indic{e^{\theta\lambda(r)}> \tilde{X}} \mbigl( \sum_{i\ge 1}\sum_{j \neq i}p_i^{\ws}p_j^{\ws}\mr)\nu(\dd \pp),
      \]
      \[
      C_0 := -\int_{\Pp\setminus \Pp_1} \Indic{e^{\theta\lambda(r)}>X} \Indic{e^{\theta\lambda(r)}> \tilde{X}} \mbigl( \sum_{i\ge 1}\sum_{j \neq i} p_i^{\ws}p_j^{\ws}(\log p_i + \log p_j) \mr)\nu(\dd \pp),
      \]
      \[
      C_{-1} := \int_{\Pp\setminus \Pp_1} \Indic{e^{\theta\lambda(r)}>X} \Indic{e^{\theta\lambda(r)}> \tilde{X}} \mbigl( \sum_{i\ge 1}\sum_{j \neq i} p_i^{\ws}p_j^{\ws}\log p_i\log p_j\mr)\nu(\dd \pp).
      \]
         By the inequality $|\log y| \le \epsilon^{-1} y^{-\epsilon}$ for $y\in [0,1]$, there is
                     \[  \Indic{e^{\theta\lambda(r)}>\tilde{X}}  \sum_{i\geq 1}\sum_{j\neq i}p_i^{\ws}p_j^{\ws}\log p_i\log p_j
        \le \epsilon^{-2}\Indic{e^{\theta\lambda(r)}>\tilde{X}}  \sum_{i\geq 1}\sum_{j\neq i}p_i^{\ws-\epsilon}p_j^{\ws-\epsilon}
   	\le     
      \epsilon^{-2}  \Indic{e^{\theta\lambda(r)}>\tilde{X}}  (\tilde{X}+2 p_1^{\ws-\epsilon}) \tilde{X}.
      \]
	We also note that $p_1^{\ws-\epsilon}\le 1 $. 
	It follows that 
	\[C_{-1}
        \le 
      \epsilon^{-2}    \mbigl(2+e^{\theta\lambda(r)} \mr) \int_{\Pp\setminus \Pp_1} \tilde{X} \, \nu(\dd \pp) .
      \]
	As $\ws-\epsilon\in \dom \kappa$, by \eqref{eq:dom} we have  
	$C_{-1}\le c_{-1} (e^{\theta\lambda(r)} + 1)$, with $c_{-1}>0$ a finite constant. 
      Similarly, we can prove that $C_0\le c_0 (e^{\theta\lambda(r)}+1)$ and $C_{1}\le c_{1} (e^{\theta\lambda(r)}+1)$, with $c_0,c_{1}>0$ finite constants. 
      Take $\bar{c}:= \max(c_1, c_0, c_{-1})>0$, then 
        \[
        	 \lQDc \mbigl[ B_t  \mr] 
        	\leq \bar{c}e^{a\ws} \int_{[0,t]} \dd r \lQDc\mBigl[( e^{- (\ws-\theta)\lambda(r)} +e^{- \ws\lambda(r)})\Indic{r< \zeta }  \mbigl(\lambda(r)+ 1 +\lambda(r)^{-1} \mr) \mr]. 
		   \]
      By \autoref{l:lambda-int}, we deduce that
      \[
 \liminf_{t\to \infty}        \lQDc[  B_t  ]<\infty.
      \] 	     
      Using similar arguments, that we omit for conciseness, we can deduce that $ \liminf_{t\to \infty}        \lQDc [ E_t  ]<\infty$. 
      By \autoref{l:lambda-int}, we also see that 
       $  \liminf_{t\to \infty}\lQDc [   A_t  ]<\infty. $ 	   
   Then Fatou's lemma yields 
   \[
   \liminf_{t\to \infty}   (A_t + B_t +E_t) <\infty, \qquad  \qquad \lQDc-a.s. 
   \]
       
        \skippar
      We next estimate $D_t$. To this end, let us consider 
      \[
      H_{\infty}:=\int_{[0,\infty)\times(0,1)\times\NN\times\Pp}
       \Indic{e^{\theta\lambda(r-)}\le X}\Indic{i\ne 1} \, N(\dd r,\dd y, \dd i, \dd \pp).
      \]
      Using again \autoref{l:cN}, \autoref{d:eta} and the inequality $|\log y| \le \epsilon^{-1} y^{-\epsilon}$ for $y\in [0,1]$, we deduce that
\[
\lQDc \mbigl[  H_{\infty} \mr]
\le \int_{\Pp} \mbiggl( 
				\int_{[0,\infty)} \lQDc \mBigl[ 
 					\Indic{e^{\theta\lambda(r)}\le X}	
 					\mbigl( X + \lambda(r)^{-1} \epsilon^{-1}\tilde{X} \mr) 
  				\mr] \, \dd r 
  			\mr) \nu(\dd \pp). 
\]
By similar arguments as in the proof of \autoref{l:lambda-int}, with notations therein, we obtain that 
\[
\int_{[0,\infty)} \lQDc \mBigl[ 
 					\Indic{e^{\theta\lambda(r)}\le X}	\mbigl(  X + \lambda(r)^{-1} \epsilon^{-1}\tilde{X}     \mr) 
  				\mr] \, \dd r 
  				= k \int_0^{\theta^{-1}\log_+ X} \frac{y}{a}\mbigl(  X + y^{-1} \epsilon^{-1} \tilde{X}     \mr)  (\WW(y)-\WW(y-a)) \, \dd y.
\]
we know that
      $\WW(y)-\WW(y-a) \to ac/{m_+}$ as $y \to \infty$, where $m_+$ is the mean of the ascending
      ladder height process of $\lambda$ and $c$ is a meaningless constant. 
      This implies that, there exists a constant $C_3$ large enough, such that
      \[
 \lQDc \mbigl[ H_{\infty} \mr] \le
 C_3 \int_{\Pp} \mBigl(  
  \theta^{-2}   X   ( \log_+ X )^2 +  \theta^{-1}   \tilde{X}  \log_+ X 
 \mr) \, \nu(\dd \pp).
      \]      
       Since $\tilde{X}  \log_+ X \le (X  \log_+ X + \tilde{X}  \log_+ \tilde{X})$,  by \eqref{e:h-dm} the right-hand-side of the above expression is finite.
       Hence $H_{\infty}$ is $\lQDc$-a.s. finite, which yields that
       $\sup_{t\ge 0} D_t<\infty$ holds $\lQDc$-a.s.
       Indeed, 
\[
       \sup_{t\ge 0} D_t \le \int_{[0,\infty) \times(0,1)\times\NN\times\Pp}
e^{a\ws  - \ws\lambda(r-)} 
\Big| \sum_{j \ne i} (\lambda(r-)- \log p_j)  p_j^{\bar{\omega}} \Big| \, \Indic{e^{\theta\lambda(r-)}\le X}\Indic{i\ne 1} N(\dd r,\dd y, \dd i, \dd \pp).\\
\]
The right-hand-side is an integral over a random point measure, whose total mass is $H_{\infty}$. 
So the fact that $H_{\infty}$ is $\lQDc$-a.s. finite yields that the integral is $\lQDc$-a.s. a finite sum. 

       In the same manner, we can also deduce that $\sup_{t\ge 0} D'_t<\infty$ holds $\lQDc$-a.s.
       This would require that $\int_{\Pp} X  (\log_+ \tilde{X})^2 \, \nu(\dd \pp)<\infty$, which is again a consequence of \eqref{e:h-dm}. 
      Having assumed \eqref{e:Q=P}, this completes the proof of \eqref{e:liminf}.  
      
      \skippar
      It remains to justify \eqref{e:Q=P}. 
      We first consider $\lEQc \mbigl[ \MD_a(\ws,t) \mm\vert \GG_{\infty} \mr]$. 
      Recall that $\Zz$ under $\lQc$ can be described as a decorated spine process as in \autoref{d:decorated-spine}. With notation therein, we have $\lambda(t)= a - \xi(t) + t\kappa'(\ws)$. Notice that each $\Zz_u(t)$ with $u\in \tree_{t}$ corresponds bijectively to a $\Zz^{[r,j]}_v(t-r)$ with $v\in \tree^{[r,j]}_{t-r}$, such that
      \[
      \Zz_u(t) = \Zz^{[r,j]}_v(t-r) + \xi(r-)+\log p_j = \Zz^{[r,j]}_v(t-r) - \lambda(r-)+ r\kappa'(\bar{\omega}) + a +\log p_j.
      \]
      Replacing $\Zz_u(t)$ in $\MD_a(\ws,t)$ by the right-hand-side of the identity above, conditioning to $\GG_{\infty}$ and using the fact that $\ws \kappa'(\ws) =\kappa(\ws)$, we have that
      \begin{eqnarr*}
\eqnarrLHS{ \lEQc\mbigl[\MD_a(\ws,t)\mm\vert \GG_{\infty} \mr] -\lambda(t) e^{a\ws  - \ws\lambda(t)}\Indic{t<\zeta}}
&=& \int_{[0,t]\times(0,1)\times\NN\times\Pp}
     e^{a\ws  - \ws\lambda(r-)} \Indic{r<\zeta}  \sum_{j \ne i} \lEQc\mbigl[\MD^{[r,j]}(a^{[r,j]}, \ws,t-r)\mm\vert \GG_{\infty} \mr]  p_j^{\bar{\omega}} \, N(\dd r,\dd y, \dd i, \dd \pp), 
      \end{eqnarr*}
      where $\zeta := \inf\{s\geq 0: \lambda(s)\leq 0\}$, $a^{[r,j]}:=\lambda(r-)- \log p_j$, and $\MD^{[r,j]}(a^{[r,j]}, \ws,t-r)$ denotes the stopped derivative martingale of the branching L\'evy process $\Zz^{[r,j]}$, i.e.  
\[\sum_{v\in \tree^{[r,j]}_{t-r}} \mbigl( a^{[r,j]} +(t-r) \kappa'(\bar\omega) -\Zz^{[r,j]}_v(t-r)  \mr) e^{-(t-r) \kappa(\bar\omega) + \bar\omega\Zz^{[r,j]}_v(t-r)}\Indic{ -s \kappa'(\bar\omega) + \Zz^{[r,j]}_{ \Anc(s,v)}(s)<a^{[r,j]}, \forall s\in[0,t-r]}.
\]
By the independence of $\Zz^{[r,j]}$ and $\GG_{\infty}$, we have the identity  \[\lEQc\mbigl[\MD^{[r,j]}(a^{[r,j]}, \ws,t-r)\mm\vert \GG_{\infty} \mr] = a^{[r,j]}=\lambda(r-)- \log p_j. \]
Summarizing, we have that $\lEQc \mbigl[\MD_a(\ws,t)\mm\vert \GG_{\infty} \mr]$ is equal to  $S_t$ as in \eqref{e:MD-g}. 

We can now prove \eqref{e:Q=P}. 
For every $s>t$, let $\GG_{s}:= \sigma(\Zz_{U(r)}(r),U(r), r\in [0,s]) \subset \hat{\FF}_{s}$. 
By the change of measure \eqref{e:cm-lambda}, for every $A\in \GG_{s}$ we have that 
 \begin{eqnarr*}
\lQDc \mbigl[\MD_a(\ws,t) \Ind_{A} \mr]
& =&a^{-1} \lEQc\mbigl[ \MD_a(\ws,t) \lambda(s)\Indic{s<\zeta}\Ind_{A}  \mr]\\
&=&a^{-1}\lEQc\mBigl[ \lEQc\mbigl[\MD_a(\ws,t)\mm\vert \GG_{\infty} \mr]\lambda(s)\Indic{s<\zeta}\Ind_{A} \mr],  \\
&=& a^{-1}\lEQc\mBigl[ S_t\lambda(s)\Indic{s<\zeta}\Ind_{A} \mr],
 \end{eqnarr*}
where the second equality is due to the fact that $\lambda(s)\Indic{s<\zeta}$ is also $\GG_{\infty}$-measurable. Since \eqref{e:MD-g} shows that $S_t$ is $\hat{\FF}_t$-measurable, it is also $\hat{\FF}_s$-measurable for $s>t$. Using again the change of measure \eqref{e:cm-lambda}, we have
\[\lQDc\mbigl[\MD_a(\ws,t)\mm\vert \GG_{s} \mr]=S_t, \qquad \lQDc-a.s.\]
 Letting $s\to \infty$, L\'evy's zero-one law leads to \eqref{e:Q=P}. 
\end{proof}

\subsection{Proof of \autoref{t:DerMart}}\label{s:proofD}
We are now ready to prove \autoref{t:DerMart}. We tackle each part separately.

\bigskip	
(i) 
    For every $a>0$, it is clear that $\MD(\omega,t)$ is equal to $a \MA(\omega, t)-\MD_a(\omega,t)$ for all $t\ge 0$ in the event
    \[ B_{\omega, a}:= \mbigl\{    \sup_{t\geq 0}\mbigl(\sup_{u\in\tree_t} \Zz_u(t)- \kappa'(\omega) t \mr)<a \mr\}.\] 
    We know from \citet[Theorem 2.3(ii)]{Dadoun:agf}\footnote{Though \citet[Theorem 2.3]{Dadoun:agf} has an extra condition in the form of his equation~(2.7), it is only required for the proof of part (i) of that theorem, and not part (ii). Part (ii), which we use here, still holds under the broader assumptions that we make.} or \cite[Theorem~1.1]{BM-mcbLp} that for the additive martingale the following convergence holds $\lP$-almost surely:  
    \begin{equation}\label{e:am}
    \lim_{t\to \infty} \mg{\omega}{t} = 0, \qquad \text{for any}~\omega\ge \ws. 
    \end{equation}
    Then $\MD(\omega,t)$ converges to a finite non-positive limit  
    \begin{equation}\label{e:Ba}
    \MD(\omega,\infty) := -\MD_a(\omega,\infty), \text{ in the event } B_{\omega,a}. 
    \end{equation}
    On the other hand, since 
    \[
    \sup_{u\in\tree_t} \Zz_u(t)- \omega^{-1}\kappa(\omega) t \leq\omega^{-1} \log   \mg{\omega}{t}, \qquad \text{for every } t>0,
    \]
    and $\kappa'(\omega)\ge \omega^{-1}\kappa(\omega)$, letting $t\to \infty$, we deduce that
    \begin{equation}\label{e:Z1}
    \lim_{t\to \infty} \mbigl( \sup_{u\in\tree_t} \Zz_u(t)- \kappa'(\omega) t\mr)= -\infty,\qquad \lP-\text{almost surely}.
    \end{equation} 
    Then $\lP \mbigl(\lim_{a\uparrow \infty} B_{\omega,a}  \mr)= 1$ as a consequence. We hence conclude that $\MD(\omega,t)$ converges $\lP$-almost surely to a finite non-positive limit. 
     
    \bigskip
    (ii)
    As $\omega>\ws$, it follows from \eqref{e:Ba} and \autoref{l:MDa} that $\MD(\omega,\infty)=0$ in $B_{\omega,a}$ for every $a>0$. Since $\lP \mbigl(\lim_{a\uparrow \infty} B_{\omega,a}  \mr)= 1$, we deduce that $\MD(\omega,\infty) = 0$ holds $\lP$-almost surely. 
    
    \bigskip
    (iii)
	For every $a>0$, we observe from \eqref{e:Z1} that $\lP$-almost surely
      \[
      \liminf_{t\to \infty} \inf_{u\in \tree_t} \mBigl[\mbigl( a +t \kappa'(\ws) -\Zz_u(t)  \mr) e^{-t \kappa(\ws) + \ws\Zz_u(t)} \mr]\geq 0,\]
      which entails that $\lP$-almost surely
      \[\lim_{t\to \infty}  (a\MA(\ws,t) - \MD(\ws,t)) \geq\lim_{t\to \infty} \MD_a(\ws,t).\] 
      We hence deduce from \eqref{e:am} and \autoref{l:MDa} that 
      \[ \lE[\MD(\ws,\infty)] \leq  \lE[-\MD_a(\ws,\infty)] = -a. \]
      Since $a>0$ is arbitrary, we readily have $ \lE[\MD(\ws,\infty)] = -\infty$.   
      
      It remains to prove that $\MD(\ws,\infty) <0$, $\lP^*$-almost surely. The following arguments are modified from the proof of Proposition 8~(iii) in \cite{BR-disc}. For every $v \in \tree$ and $t\geq 0$, denote 
      \[\MD^{(v)}(\ws,t):=  e^{-t\kappa(\ws)}
      \sum_{u\in \tree_{t+1}, v \preceq u, b_u > 1} \mbigl(\Zz_u(t+1) -\Zz_v(1) - t\kappa'(\ws)\mr) e^{\ws (\Zz_u(t+1) - \Zz_v(1))}
      \]
and
\[\MA^{(v)}(\ws,t):= e^{-t\kappa(\ws)}\sum_{u\in \tree_{t+1}, v \preceq u, b_u > 1}  e^{\ws (\Zz_u(t+1)-\Zz_v(1))},\]
with convention $\MD^{(v)}(\ws,t)= \MA^{(v)}(\ws,t) = 0$ whenever $v\not\in \tree_1$.  
      Then we have the following decomposition: 
      \begin{equation}\label{e:t+1}
      \MD(\ws,t+1)= e^{-\kappa(\ws)}\mbiggl( \sum_{v\in \tree_1} e^{\ws \Zz_v(1) }\MD^{(v)}(\ws,t)
      + \sum_{v\in \tree_1} e^{\ws \Zz_v(1) }(\Zz_v(1) -\kappa'(\ws) ) \MA^{(v)}(\ws,t) \mr).
      \end{equation}

      Let us start with proving the following technical result: 
      \begin{equation}\label{e:cvto0}
        \sum_{v\in \tree_1} e^{\ws \Zz_v(1) }(\Zz_v(1) -\kappa'(\ws) ) \MA^{(v)}(\ws,t) \underset{t\to \infty}{\longrightarrow} 0  \quad \text{ in probability with respect to }\lP.
      \end{equation}
      Let $c>0$ be small enough such that $\ws-c\in \dom \kappa$. One observes that $|\log y| \leq  c^{-1} (y^{-c} + y^{c})$ for every $y>0$, then
      \begin{multline*}
			\lE \mbiggl[ \sum_{u\in \tree_1} |\Zz_u(1)-\kappa'(\ws)| e^{\ws \Zz_u(1)} \mr] 
            \le c^{-1}\lE \mbiggl[ \mBigl(\sum_{u\in \tree_1} \mbigl( e^{(\ws+c) \Zz_u(1)} + e^{(\ws-c) \Zz_u(1)}+ |\kappa'(\ws)| e^{\ws \Zz_u(1)} \mr) \mr)  \mr]. 
      \end{multline*}
      The second expectation is finite, so is the first one.  
      Fix an enumeration of $\tree$ and denote for every $u\in \tree$ its index by $I_u\in \NN$. Then for every $\epsilon, \delta>0$, there exists $n_0\in \NN$, depending on $\epsilon$ and $\delta$, such that 
      \[  \lE \mbiggl[ \sum_{v\in \tree_1, I_v>n_0} |\Zz_v(1)-\kappa'(\ws)| e^{\ws \Zz_v(1)} \mr] \leq \delta \epsilon. 
      \]
Furthermore, by conditioning on $\FFs_1$ and using the branching property, \autoref{l:Zbar-bp}, we deduce the identity
     \[  \lE \mbiggl[ \sum_{v\in \tree_1, I_v>n_0} |\Zz_v(1)-\kappa'(\ws)| e^{\ws \Zz_v(1)} \MA^{(v)}(\ws,t) \mr]= \lE \mbiggl[ \sum_{v\in \tree_1, I_v>n_0} |\Zz_v(1)-\kappa'(\ws)| e^{\ws \Zz_v(1)} \mr]. 
      \]
      Then an application of Markov inequality yields 
      \[  \lP \mbiggl[ \sum_{v\in \tree_1, I_v>n_0} |\Zz_v(1)-\kappa'(\ws)| e^{\ws \Zz_v(1)} \MA^{(v)}(\ws,t) > \delta \mr] \leq  \epsilon. 
      \]
      Moreover, we see from \eqref{e:am} that each $\MA^{(v)}(\ws,t)$ converges $\lP$-almost surely to $0$. It follows that
      \[  \sum_{v\in \tree_1, I_v\leq n_0} |\Zz_v(1)-\kappa'(\ws)| e^{\ws \Zz_v(1)}\MA^{(v)}(\ws,t) \underset{t\to\infty}{\longrightarrow} 0, \qquad \lP\text{-a.s. }
      \]
      Hence we have proved \eqref{e:cvto0}. 
   
We now go back to \eqref{e:t+1}. By the branching property, $(\MD^{(v)}(\ws,\cdot), v\in \tree_1)$ are independent copies of $\MD(\ws, \cdot)$, also independent of $\FFs_1$. Then we see from part~(i) that each $\MD^{(v)}(\ws,t)$ converges $\lP$-almost surely to a non-positive limit $\MD^{(v)}(\ws,\infty)$, which has the same law as $\MD(\ws,\infty)$. Letting $t\to \infty$ in \eqref{e:t+1} and using \eqref{e:cvto0}, we deduce that, for every $u\in \tree$,
\[
\MD(\ws,\infty)= \Indic{u\in \tree_1} e^{-\kappa(\ws)}e^{\ws Z_{u}(1) }\MD^{(u)}(\ws,\infty)+R,\qquad  \lP\text{-a.s.},
\]
where \[R := \lim_{t\to \infty}e^{-\kappa(\ws)}\sum_{v\in \tree_1, v\ne u} e^{\ws \Zz_v(1) } \MD^{(v)}(\ws,t), \quad \text{ in probability with respect to }\lP.
\]
Since $\MD^{(u)}(\ws,\infty)$ is independent of $R$, the above identity entails that 
\[
\lP \mbigl( \MD(\ws,\infty) > 0 \mr)
\geq \rho \cdot \lP \mbigl( R >0\mr) , 
\]
where $\rho:=  \lP \mbigl( \MD(\ws,\infty) = 0 \mr) = \lP \mbigl( \MD^{(u)}(\ws,\infty) = 0 \mr)$.

Recall from part~(i) that $\MD(\ws,\infty)$ is non-positive, i.e. $\lP \mbigl( \MD(\ws,\infty) > 0 \mr) =0$. Suppose now that
$\rho >0$ (otherwise there is nothing to prove), then $\lP \mbigl( R > 0 \mr) = 0$. It follows that
\[
\MD(\ws,\infty)\leq \Indic{u\in \tree_1} e^{-\kappa(\ws)}e^{\ws Z_{u}(1) }\MD^{(u)}(\ws,\infty) ,\qquad  \lP\text{-a.s.}
\]
This inequality holds for every $u\in \tree$, we hence deduce by the independence of the family $\big( \MD^{(u)}(\ws,\infty), u\in \tree\big)$ that 
\begin{equation}\label{e:extinct}
\rho \leq \lE\mbigl[ \rho^{\# \tree_1}\mr] , \end{equation}
with $\# \tree_1$ the number of particles at time $1$. It has been proved in (ii) that
$\lE \mbigl[ \MD(\ws,\infty) \mr] = -\infty$, so we also have $\rho <1$. 
Hence \eqref{e:extinct} entails that $\rho =\lP \mbigl( \MD(\ws,\infty) = 0 \mr)$ is smaller or equal to the extinction probability.    
On the other hand, it is clear that the extinction event is included in $\{\MD(\ws,\infty) = 0\}$.  
We conclude that $\MD(\ws,\infty) < 0$, $\lP^*$-almost surely.

  This completes the proof.

\section*{Acknowledgements} 
We would like to thank Jean Bertoin and Andreas E. Kyprianou for their helpful comments
on earlier drafts of this work.
Q.S.\ was supported by math--STIC of Paris University 13 and the SNSF fellowship P2ZHP2\_171955. 
This work has been much improved by the careful reading and comments of two
anonymous referees.

  \bibliography{master}
  
\end{document}